\documentclass{amsart}[14]
\usepackage{amssymb}
\usepackage{graphicx}
\usepackage[english]{babel}
\usepackage[pdftitle={NonDiv Wave},
  pdfauthor={Doghonay},
  pdffitwindow=true,
  breaklinks=true,
  colorlinks=true,
  urlcolor=blue,
  linkcolor=red,
  citecolor=red,
  anchorcolor=red]{hyperref}
\usepackage[numbers]{natbib}

\numberwithin{equation}{section}
\numberwithin{table}{section}
\numberwithin{figure}{section}

\newcommand{\e}{\varepsilon}
\newcommand{\ue}{u^{\varepsilon}}
\newcommand{\ve}{v^{\varepsilon}}
\newcommand{\ueeta}{u^{\varepsilon,\eta}}

\newcommand{\bx}{{\bf x}}

\newcommand{\by}{{\bf y}}
\newcommand{\bzeta}{{\bf{\zeta}}}
\usepackage{enumerate}

\newtheorem{remark}{Remark}
\newtheorem{theorem}{Theorem}
\newtheorem{corollary}{Corollary}

\newtheorem{definition}{Definition}
\newtheorem{lemma}{Lemma}

\begin{document}

\title[Multiscale Methods for the Wave Equation]{An Equation-Free Approach for Second Order Multiscale Hyperbolic Problems in Non-Divergence Form}

\author[D.~Arjmand]{Doghonay Arjmand}
\author[G.~Kreiss]{Gunilla Kreiss}

\address{ANMC, Section de Math{\'e}matiques, {\'E}cole Polytechniques F{\'e}d{\'e}rale de Lausanne, Station 8, CH-1015 Lausanne, Switzerland (doghonay.arjmand@epfl.ch)}

\address{Division of Scientific Computing \\
  Department of Information Technology \\
  Uppsala University \\
  SE-751 05 Uppsala, Sweden, (gunilla.kreiss@it.uu.se).}




\maketitle

          \begin{abstract}
              The present study concerns the numerical homogenization of second order hyperbolic equations in non-divergence form, where the model problem includes a rapidly oscillating coefficient function. These small scales influence the large scale behavior, hence their effects should be accurately modelled in a numerical simulation. A direct numerical simulation is prohibitively expensive since a minimum of two points per wavelength are needed to resolve the small scales. A multiscale method, under the equation free methodology, is proposed to approximate the coarse scale behaviour of the exact solution at a cost independent of the small scales in the problem. We prove convergence rates for the upscaled quantities in one as well as in multi-dimensional periodic settings. Moreover, numerical results in one and two dimensions are provided to support the theory.
          \end{abstract}

          \section{Introduction}\label{intro}
 Various engineering applications e.g., from  seismology, medical imaging, or material science require simulations of the wave equation in heterogeneous media. In general, the model problems may be in divergence or non-divergence form having different \emph{homogenized limits} (as the wavelength of the heterogeneities tends to zero), and hence multiscale methods need to be developed depending on the structure of the problem. Multiscale methods for the second order wave equation in divergence form have been developed and analyzed in the past, see e.g., \cite{Abdulle_Grote_1,Engquist_Holst_Runborg_1,Arjmand_Runborg_3,Arjmand_Stohrer}. In the present work, we consider a second-order scalar wave equation in non-divergence form 
\begin{equation} \label{Main_MultiscaleWave_Eqn}
\left\{
\begin{array}{lll}
 \partial_{tt} \ue(t,\bx) =    \displaystyle \sum_{i,j=1}^{d} A_{ij}^{\e}(\bx) \partial_{x_i x_j} \ue(t,\bx) + f(t,\bx), \quad \text{ in } (0,T]\times \Omega \\
\ue(0,\bx) = g(\bx), \quad \partial_t \ue(0,\bx) = h(\bx), \quad \text{ on } \{ t=0 \}\times \Omega, \\
\ue(t,\bx)   = 0, \quad \text{ on } \partial \Omega,
\end{array}
\right.
\end{equation} 
where $\Omega$ is a bounded open subset in $\mathbb{R}^d$ with $|\Omega| = O(1)$, and $A^{\e}$ is a bounded symmetric positive-definite matrix function in $\mathbb{R}^{d \times d}$ such that for every $\bzeta \in \mathbb{R}^{d}$
\begin{equation} \label{Assump_Coeff}
c_1 |\bzeta|^{2} \leq \sup_{\bx \in \Omega} \zeta^{T} A^{\e}(\bx) \bzeta \leq c_2 \left| \bzeta \right|^{2}, \quad \text{ and } A_{ij}^{\e} = A_{ji}^{\e}.
\end{equation}
The homogeneous boundary condition in \eqref{Main_MultiscaleWave_Eqn} is assumed only for simplicity and other well-posed boundary conditions can be treated similarly. The parameter $\e \ll 1$ represents the wavelength of the small scale variations in the media, and $T=O(1)$ is a constant independent of $\e$. 

When $\e \ll 1$, a direct numerical approximation of \eqref{Main_MultiscaleWave_Eqn} is very expensive since the rapid variations in $A^{\e}$ must be represented over a much larger computational domain. In such a case, the tendency is to instead look for an effective or a homogenized solution $u^{0}$ which does not depend on the small scale parameter $\e$. Analytically, this is related to the theory of homogenization, see e.g., \cite{Bensoussan_Lions_Papanicolaou,Cioranescu_Donato,Jikov_Kozlov_Oleinik}, where the goal is to replace the oscillatory coefficient $A^{\e}$ by a slowly varying coefficient $A^{0}$ and solve for the corresponding homogenized solution $u^0$ at a cost independent of $\e$. Mathematically speaking, the homogenized solution $u^{0}$ is obtained in the limit $u^{0} = \lim_{\e \to 0 } u^{\e}$ (this convergence is understood as weakly-* in $L^{\infty}(0,T;L^2(\Omega))$), see e.g., \cite{Bensoussan_Lions_Papanicolaou,Cioranescu_Donato}. For example, when the medium is periodic such that $A^{\e}(\bx) = A(\bx/\e)$ where $A$ is a $Y:=[0,1]^d$-periodic function, $u^{\e}$ converges to a limit solution $u^{0}$ (as $\e \to 0$) which solves 
\begin{equation}  \label{Effective_Eqn}
\left\{
\begin{array}{ll}
\partial_{tt} u^0(t,\bx) = \displaystyle \sum_{i,j=1}^{d} A^{0}_{ij} \partial_{x_i x_j} u^0(t,\bx)+ f(t,\bx), \quad \text{ in } (0,T]\times \Omega \\
u^0(0,\bx) = g(\bx), \quad \partial_t u^0(0,\bx) = h(\bx), \quad \text{ on } \{ t=0 \}\times \Omega. 
\end{array}
\right.
\end{equation} 
Here  the homogenized coefficient $A^{0}$ is a constant matrix given by
\begin{equation}\label{A0_Eqn} 
A^{0} = \int_Y A(\by) \rho(\by) \; d\by,
\end{equation}
and $\rho$ solves the equation
\begin{equation} \label{Invarient_Dist_Eqn}
\left\{ 
\begin{array}{ll}
\displaystyle  -\sum_{i,j=1}^{d} \partial_{y_i y_j} \left( A_{ij}(\by) \rho(\by) \right) = 0, \text{ in } Y=[0,1]^d, \\
\int_{Y} \rho(\by) \; d\by  = 1, \quad \rho \text{ is } Y\text{-periodic}.
\end{array}
\right.
\end{equation}

\begin{remark} \label{Rem_Hom_Div}When the main problem \eqref{Main_MultiscaleWave_Eqn} is in divergence form, i.e., 
\begin{equation} \label{Main_MultiscaleWaveDiv_Eqn}
\left\{
\begin{array}{ll}
 \partial_{tt} \ve(t,\bx) =    \nabla \cdot \left(  A^{\e}(\bx) \nabla \ve(t,\bx)  \right) + f(t,\bx), \quad \text{ in } (0,T]\times \Omega \\
\ve(0,\bx) = g(\bx), \quad \partial_t \ve(0,\bx) = h(\bx), \quad \text{ on } \{ t=0 \}\times \Omega,
\end{array}
\right.
\end{equation} 
the corresponding homogenized equation reads as

\begin{equation}  \label{Eqn_Intro_DivHomogenization}
\left\{
\begin{array}{ll}
\partial_{tt} v^0(t,\bx) &= \nabla \cdot   \Big( A_{div}^{0}(\bx) \nabla v^0(t,\bx) \Big) + f(t,\bx), \quad \text{ in } (0,T] \times \Omega   \\
v^0(0,\bx) &= g(\bx), \quad \partial_t v^0(0,\bx) = h(\bx), \quad \text{ on } \{ t=0 \} \times  \Omega. 
\end{array}
\right.
\end{equation} 
If the medium is additionally periodic such that $A^{\e}(\bx) = A(\bx/\e)$ for a $Y$-periodic matrix function $A$, the homogenized coefficient $A^{0}_{div}$ is a constant matrix given by
\begin{equation}\label{eqn_A0Div} 
[A^{0}_{div}]_{ij} = \int_Y \left( A_{ij}(\by)  + \sum_{k=1}^{d} A_{ik} \partial_{y_k} \chi_{j}(\by) \right) \; d\by,
\end{equation}
and $\{ \chi_{j} \}_{j=1}^{d}$ are $Y$-periodic solutions of the following set of cell problems
\begin{align*} \label{eqn_CellProblemChi}
\nabla_{\by} \cdot \left(  A(\by) \nabla_{\by} \chi_{j}(\by)    + A(\by) e_{j}  \right) = 0, \quad \int_{Y} \chi_{j}(\by) \; d\by = 0,
\end{align*}
where $\{  e_{j} \}_{j=1}^{d}$ are canonical basis vectors in $\mathbb{R}^{d}$. 
\end{remark}

The main drawback of analytical homogenization is that explicit formulas for the homogenized matrix $A^{0}$ are available only in a few academic cases of interests such as the periodic case \eqref{A0_Eqn}. To treat more realistic scenarios, e.g., where slow and fast variations (not particularly periodic) are allowed at the same time, several general purpose multiscale approaches were proposed over the last two decades. Variational multiscale methods (VMM) due to Hughes et al. \cite{Hughes_Feijoo_Mazzei_Quincy}, multiscale finite element methods (MsFEM) due to T. Hou et al. \cite{Hou_Wu}, heteregeneous multiscale methods (HMM) due to E and Engquist \cite{E_Engquist_1}, and the equation free approach due to Kevrikidis et al. \cite{Kevrikidis_etal} are among such successful examples. The overall goal behind such strategies is to approximate the solution $u^{\e}$ (or $u^{0}$) with no a priori knowledge about the structure of $A^{\e}$ or the homogenized coefficient $A^{0}$. Several multiscale methods have been designed and analysed under the above-mentioned frameworks. Without being exhaustive, we refer to \cite{Abdulle,Babuska,Gloria_2,Henning_Malqvist_1,Arjmand_Runborg_1,Hou_Wu,Efendiev_Hou_Book} for applications to elliptic problems, see \cite{Abdulle_Vilmart_1,Arjmand_Stohrer,Samaey_Roose_Kevrikidis_2005} for applications to parabolic problems, and \cite{Abdulle_Grote_1,Abdulle_Henning_1,Arjmand_Runborg_2,Arjmand_Runborg_3,Engquist_Holst_Runborg_1} for applications to second order wave equations.  Other alternative approaches are wavelet based numerical homogenization due to Engquist and Runborg  \cite{Engquist_Runborg_1}, and the harmonic coordinate transformations due to Owhadi et al. \cite{Owhadi_Zhang_Berlyand,Owhadi_Zhang_1}.

In general, the homogenized limits for the wave equation in non-divergence and divergence forms are different from each other, cf. \eqref{Effective_Eqn} and \eqref{Eqn_Intro_DivHomogenization}. Multiscale methods are typically designed based on some assumptions about the form of the homogenized equation (although homogenized parameters may not be known explicitely). This makes the multiscale modelling of the the wave equation in divergence and non-divergence form differ from each other. Moreover, the analysis for the multiscale methods for the wave equation in divergence form typically exploits the symmetry properties of the operator $-\nabla \cdot \left(  A^{\e} \nabla  \right)$, which is missing for problems in non-divergence form.

In the present article, we develop and analyse an equation free type multiscale approach for a numerical approximation of the wave equation \eqref{Main_MultiscaleWave_Eqn}. The general idea behind the equation free approach (EFA) is to assume a coarse scale model of the form $\partial_{tt} U = F(U,\nabla U, \nabla^2 U, \ldots)$, and compute (upscale) $F$ locally by simulating the original multiscale problem in small domains with a size comparable to the size of the smallest scale in the PDE. While doing this, the microscopic problems are also provided with the coarse scale data, $i.e., U, \nabla U, \nabla^2 U, \ldots$. Therefore, the coupling between the microscale and the macroscale model should be understood as a two-way coupling. The efficiency of the method comes from the fact that, multiscale problems of the form \eqref{Main_MultiscaleWave_Eqn} are solved only in small temporal and spatial domains, while the method still retains a good approximation of the overall macroscopic behavior. The main requirement for the EFA is the assumption of scale separation; namely that the wavelength, $\e \ll 1$, of the microscopic variations is much smaller than the size of the computational geometry (which is assumed to be $O(1)$ in this paper). Moreover, the generality of the method is due to the fact that no knowledge (other than the assumption of scale separation) about the properties of the media, or the precise value of the small scale parameter $\e$ are assumed.

Although the equation free approach has been developed in the context of numerical homogenization for parabolic problems and hyperbolic conservation laws, see e.g. \cite{Samaey_Kevrikidis_Roose_2006,Samaey_Roose_Kevrikidis_2006,Samaey_Roose_Kevrikidis_2005}, not much of attention has been given to applications to the second order wave equation. Conceptually, the method developed in the present work is similar to the HMM-based multiscale numerical methods \cite{Abdulle_Grote_1,Engquist_Holst_Runborg_1,Arjmand_Runborg_3}, for problems in divergence form, but with a few changes in the way the micro- and the macromodels are coupled, the form of the macroscopic and microscopic equations, and the upscaling procedure, see Section \ref{HMM_Sec}.  From an analysis point of view, the previous theoretical results rely on the symmetry property of the operator $-\nabla \cdot A^{\e} \nabla$ and hence the previous theories can not be directly used for the wave equation \eqref{Main_MultiscaleWave_Eqn} due to the breakdown of the symmetry of the operator. This paper aims at generalizing the previous analysis, which is valid only for symmetric operators, to non-symmetric operators of type   \eqref{Main_MultiscaleWave_Eqn}. For the analysis, we consider a periodic setting, where $A^{\e}_{ij}(\bx) = \delta_{ij} a^{\e}$ and $a^{\e}(\bx) = a(\bx/ \e)$, where $a$ is a smooth $Y$-periodic function. An analysis in one and higher ($d=2,3$) dimensions is presented. As the one dimensional theory is much simpler than the one for higher dimensions, the former is presented first. The ideas are then extrapolated and extended to higher dimensions. The periodicity assumption is used only to simplify the theory but the method itself is numerically shown to perform equally well for more complicated coefficients (e.g. almost periodic functions, and locally-periodic functions).

This paper is structured as follows. In Section \ref{HMM_Sec}, the multiscale method is presented. In Section \ref{Prelim_Sec}, a few utility results are introduced.  Section \ref{Analysis_Sec} includes the main result of this article, which is an analysis for the upscaling error. In a subsequent section, an error estimate for the difference between the fully-discrete numerical solution, see Remark \ref{Rem_FullyDiscrete}, and the exact homogenised solution is given. The last section of this article contains numerical results for one and two dimensional problems.

\section{The Multiscale Method}
\label{HMM_Sec}
The main components of the multiscale strategy proposed here are a macro- and a micromodel. The macromodel describes the coarse scale part of the solution $\ue$ to problem \eqref{Main_MultiscaleWave_Eqn}. The macromodel reads as
\begin{equation}
\label{Intro_Macro_Model}
\text{Macromodel: }
\left\{
\begin{array}{lll}
 \partial_{tt} U(t,\bx)  - F(\bx,\nabla^2 U) = f(\bx), \quad \text{ in } (0,T] \times \Omega  \\
U(0,\bx) = g(\bx), \quad \partial_t U(0,\bx) = h(\bx), \quad \text{ on } \{ t=0 \} \times \Omega \\
U(t,\bx)  = 0 \quad \text{ on }   [0,T] \times \partial \Omega. 
\end{array}
\right.
\end{equation}
Here $U$ is the macroscopic solution, $F$ is the missing quantity in the macromodel, and $\nabla^2 U$ represents all the mixed second-derivatives in $d$-dimensions. For simplicity, it is assumed that $\Omega = [0,L]^d$. A finite difference discretization of the macro problem \eqref{Intro_Macro_Model} gives
\begin{equation} \label{Macro_Solver_Eqn}
\begin{array}{ll}
U_{I}^{n+1} = 2 U_{I}^{n} - U_{I}^{n-1} + \Delta t^{2} \left(  F_{I}^{n} + f_{I}^{n} \right). 
\end{array}
\end{equation}
Here $I = (i_1,i_2,\ldots,i_d)$ is a multi-index, and $U_{I}^{n}$ represents the macroscopic solution at the point $(\bx  = \bx_I, t = t_n)$, where $\{ \bx_I  = I \triangle x\}$, with $0 \leq i_j \leq N_x$, $N_x \triangle x  = L$, and $t_n = n \triangle t$, with $N_t \triangle t  =  T$. Moreover,  $U^0_{I} = g_{I}$, and $U^{1}_{I}$ is given by 
\begin{align*}
U^{1}_{I} & \approx U(\triangle t, \bx_{I}) \approx U(0,\bx_{I}) + \triangle t \partial_t U(0,\bx_{I}) + \dfrac{\triangle t^2}{2} \partial_{tt} U(0,\bx_{I}) \\
               & \approx g_I  + \triangle t h_I + \dfrac{\triangle t^2}{2} \left(  F_I^{0}(\bx_I, \nabla^2 g_I) + f_I \right),
\end{align*}
where $U(0,\bx), \partial_t U(0,\bx)$ are directly replaced by the initial data in \eqref{Intro_Macro_Model}, and the term $\partial_{tt} U(0,\bx)$ is rewritten using the equation \eqref{Intro_Macro_Model}, which also requires computing $F$ at time $t=0$. To compute the missing quantity $F_{I}^{n}$ in the macro solver \eqref{Macro_Solver_Eqn}, we solve the multiscale problem \eqref{Main_MultiscaleWave_Eqn} over a microscopic box $I_{\tau} \times \Omega_{\bx_{I},\eta}$, where $I_{\tau} = (0,\tau/2]$ and $\tau/2$ is the final time for the microscopic simulations, and $\Omega_{\bx_{I},\eta}:= \bx_{I} + [-\ell_{\eta}, \ell_{\eta}]^{d}$ where $\ell_{\eta} \geq \frac{\eta}{2} + \frac{\tau}{2} \sqrt{|A|_{\infty}} $\footnote{The condition $\ell_{\eta} \geq \frac{\eta}{2} + \frac{\tau}{2} \sqrt{|A|_{\infty}} $ is to ensure that the boundary conditions of the micromodel \eqref{Intro_Micro_Problem} do not have any influence on the interior solution.}, and in practice $\tau =\eta = O(\e)$; see also Remarks \ref{Rem_Cost_Tau_Eps} and \ref{Rem_FiniteSpeedOfWaves}. In other words, we solve
\begin{equation}
\label{Intro_Micro_Problem}
\text{Micromodel: }
\left\{
\begin{array}{lll}
\partial_{tt}u^{\e,\eta}(t,\bx)  - \displaystyle \sum_{i,j=1}^{d} A^{\e}_{ij}(\bx) \partial_{x_i x_j} u^{\e,\eta}(t,\bx) = 0,  \text{ in }  I_{\tau} \times \Omega_{\bx_{I},\eta} \\
u^{\e,\eta}(0,\bx)  = \hat{u}(\bx), \quad \partial_{t}u^{\e,\eta}(0,\bx)= 0,  \quad \text{ on } \{ t=0 \} \times {\Omega_{\bx_{I},\eta}}, \\
u^{\e,\eta} - \hat{u} \quad \text{ is periodic on } \quad  {\Omega_{\bx_{I},\eta}},
\end{array}
\right.
\end{equation}
where $\hat{u}(\bx)$ is a  quadratic polynomial approximating the coarse scale data $U_{I}^{n}$, in the least square sense, at the point $\bx_{I}$. The choice of quadratic polynomials for $\hat{u}$ is to ensure the \emph{consistency}, see Definition \ref{Def_Consistency}, of the microscopic simulations with the current macroscale data. From a modeling point of view, the issue of consistency is known to be one of the necessary conditions for the EFA type algorithms to perform well, see e.g. \cite{E_Engquist_1,Samaey_Roose_Kevrikidis_2005,Samaey_Kevrikidis_Roose_2006,Samaey_Roose_Kevrikidis_2006}.

\begin{remark} \label{Rem_Cost_Tau_Eps}Note that if $\tau = \eta  = O(\e)$, the computational cost of solving the micro problem \eqref{Intro_Micro_Problem} becomes independent of $\e$ since the solution will contain only few oscillations, in time and space, within the microscopic domain. Moreover, typical multiscale numerical methods result in errors of the form $(\e/\eta)^{e}$ for some $e \geq 1$, see e.g.\ \cite{Engquist_Holst_Runborg_1,Arjmand_Runborg_1}, which motivates the need for the additional constraint $\eta = \tau > \e$, as otherwise we would get $O(1)$ errors. 
\end{remark}
For the local averaging we introduce the space $\mathbb{K}^{p,q}$ which consists of functions $K \in C^{q}(\mathbb{R})$ compactly supported in $[-1,1]$, and $K^{(q+1)} \in BV(\mathbb{R})$, where the derivative is understood in the weak sense and $BV$ is the space of functions with bounded variations on $\mathbb{R}$, see e.g. \cite{Engquist_Tsai_1,Arjmand_Runborg_1,Arjmand_Runborg_2} for details. Moreover, the parameter $p$ represents the number of vanishing moments
$$
\int_{\mathbb{R}} K(t) t^r dt  = 
\begin{cases}
1 & r=0, \\
0 & r \leq p.
\end{cases}
$$
As local averaging takes place in a domain of size $\eta$, we consider the scaled kernel
$$
K_{\eta}(x) = \dfrac{1}{\eta} K(x/\eta).
$$
Finally, the flux $F_{I}^{n}$ is computed by\footnote{The dependency of $F_I^n$ on the second derivative $\nabla^2 U$ comes from the fact that the micro solution $u^{\e,\eta}$ depends on $\nabla^2 U$ through the initial data.} 
\begin{equation}\label{Intro_HMM_Flux}
F_I^{n} := F(\bx_I, \nabla^2 U(t_n)) = \left( \mathcal{K}_{\tau,\eta} \ast \sum_{i,j}A_{ij}^{\e} \partial_{x_i x_j} \ueeta(\cdot,\cdot) \right)(0,\bx_{I}),
\end{equation}
where
\begin{equation*}
\left( \mathcal{K}_{\tau,\eta} \ast f \right)(t,\bx) := \int_{t-\tau/2}^{t+\tau/2}\int_{\Omega_{\bx,\eta}} K_{\eta}(\tilde{\bx}-\bx) K_{\tau}(\tilde{t}- t) f(\tilde{t}, \tilde{\bx}) \; d\tilde{\bx} \; d\tilde{t},
\end{equation*}
and where in $d$-dimension, $K_{\eta}(\bx)$ is understood as 
$$
K_{\eta}(\bx) = K_{\eta}(x_1)K_{\eta}(x_2) \cdots K_{\eta}(x_d).
$$
This completes all the steps for the EFA solution $U^{n}_I$, solving \eqref{Macro_Solver_Eqn}, to approximate the solution $u^{0}$ of the homogenised equation \eqref{Effective_Eqn}. Moreover, comparing the homogenized equation \eqref{Effective_Eqn} with the macromodel \eqref{Intro_Macro_Model}, one can see that the numerical solution $U$ will stay close to the homogenized solution $u^{0}$ if the upscaled data $F$, given in \eqref{Intro_HMM_Flux},  is close to the homogenized quantity:

\begin{equation} \label{Definition_Fhat}
\hat{F}(\bx,\nabla^{2} U)  = \sum_{i=1}^{d} A^{0}_{ij} \partial_{x_i x_j} U(\bx). 
\end{equation}
Therefore, the main part of the analysis is to give a bound for the difference $| F - \hat{F}|$.

When compared to HMM type algorithms for the wave equation, cf.\    \cite{Engquist_Holst_Runborg_1}, the multiscale algorithm described here has three main differences: $1)$ the macromodel \eqref{Intro_Macro_Model} is of the form $\partial_t U  - F = f$, while the macromodel in \cite{Engquist_Holst_Runborg_1} has the form $\partial_t U  - \nabla \cdot F = f$, $2)$ the initial data of the micromodel \eqref{Intro_Micro_Problem} is a second order polynomial while in \cite{Engquist_Holst_Runborg_1} a linear polynomial is used as initial data, $3)$  the upscaling step \eqref{Intro_HMM_Flux} contains a second derivative of the microscopic solution, while in \cite{Engquist_Holst_Runborg_1} the first derivative of the microscopic solution is used in the upscaling step. These differences are mainly due to the fact that the homogenized equation corresponding to multiscale wave equations in non-divergence form is different from that in divergence form. Moreover, the choice of the periodic boundary conditions in the micromodel \eqref{Intro_Micro_Problem} is not unique and one may also use Dirichlet boundary conditions, e.g.\ $u^{\e,\eta} = \hat{u}$, similar to the standard HMM algorithms.

\begin{definition} \label{Def_Consistency} The coarse scale data $\hat{u}(\bx)$ is called a consistent initial data (up to $O(\delta)$) for the micro problem \eqref{Intro_Micro_Problem} if  
$$\left( \mathcal{K}_{\tau,\eta} \ast \ueeta \right)(0,\bx)  = \hat{u}(\bx) + O(\delta),    \quad \text{ for all }   \bx \in I_{\tau} \times \omega_{\eta},$$
where $\ueeta$ solves the micro problem \eqref{Intro_Micro_Problem}, and $\omega_{\eta}  = \bx_I + [-\eta/2,\eta/2]^d$ is the interior region of the microscopic domain $\Omega_{\bx_I,\eta}$. 
\end{definition}

\begin{remark} Note that in the upscaling step \eqref{Intro_HMM_Flux}, we need  the values of the solution for the micro problem \eqref{Intro_Micro_Problem} in the time interval $[-\tau/2,0)$. This requires no additional cost since  the symmetry property $u^{\e,\eta}(t,\bx) = u^{\e,\eta}(-t,\bx)$ easily follows due to the condition $\partial_t u^{\e,\eta}(0,\bx) = 0$.
\end{remark}

\begin{remark}\label{Rem_FiniteSpeedOfWaves} Observe that (because of the compact support of the kernel  $K_{\eta}( \bx - \bx_I)$ in $\bx_I + [-\eta/2,\eta/2]^d$) the local averaging in the upscaling step \eqref{Intro_HMM_Flux} takes place in an interior region of $\Omega_{\bx_I,\eta}$; namely the  region $I_{\tau} \times \omega_{\eta}$, where $\omega_{\eta}  = \bx_I + [-\eta/2,\eta/2]^d$. When $\ell_{\eta} \geq \frac{\eta}{2} + \frac{\tau}{2} \sqrt{|A|_{\infty}}$, the solution $\ueeta$ to the micro problem \eqref{Intro_Micro_Problem} in the region $I_{\tau} \times \omega_{\eta}$, is not affected by the periodic boundary conditions of the micromodel \eqref{Intro_Micro_Problem}. This is due to the finite speed of propagation of waves, see e.g. \cite{Evans_Book}; i.e., the near boundary waves do not have  enough time to reach the region $\omega_{\eta}$ over the time interval $I_{\tau}$.
\end{remark}

\begin{remark}\label{Rem_FullyDiscrete} In practice, to compute a fully discrete counterpart of the EFA solution $U^n_I$, one needs to discretise the micromodel \eqref{Intro_Micro_Problem}, and the integral \eqref{Intro_HMM_Flux}. Later in the analysis, we denote this fully discrete solution by $\tilde{U}^{n}_{I}$. We assume that the micromodel is solved by a Leap frog scheme, see Section \ref{FullyDiscrete_Sec}. Moreover, for the analysis (as it is the case also for the numerical examples in this paper), we assume that a standard trapezoidal rule is used for the integration in \eqref{Intro_HMM_Flux}.
\end{remark}

\section{Preliminaries}
\label{Prelim_Sec}
The numerical method developed in the previous section is designed for treating coefficients satisfying the general conditions \eqref{Assump_Coeff}. However, The analysis will be given only for isotropic material modelled by coefficients of the form $A^{\e}(\bx) = a(\bx/\e) I $, where $a \in C^{\infty}_{Per}(Y)$ is a $Y$-periodic scalar function. In this case, the homogenized coefficient, from \eqref{A0_Eqn}, is a constant matrix and given by, see e.g. \cite{Froese_Oberman_09},
\begin{equation} \label{Eqn_HarmonicMean}
A^{0} = a^{0} I, \text{ where } a^{0}  = \left(  \int_{Y} \dfrac{1}{a(\by)} d\by \right)^{-1}.
\end{equation}
In general, the homogenised coefficient $A^{0}$ for non-divergence structures, given by \eqref{A0_Eqn}, is different than the homogenised coefficient $A^{0}_{div}$, computed by \eqref{eqn_A0Div}, for divergence structures. However, under a special theoretical setting, they are equal to each other, see Remark \ref{Rem_OneDimen_Hom}. This fact, together with the Theorem \ref{Thm_FDiv}, given below, for divergence structures will be used in a part of the analysis in one-dimension. 

\begin{remark} \label{Rem_OneDimen_Hom} In one-dimensional periodic media, the homogenized coefficient $A^0_{div}$ is the same as the homogenized coefficient $A^0$, given by the harmonic mean \eqref{Eqn_HarmonicMean}.
\end{remark}


\begin{theorem}\label{Thm_FDiv}Let $\hat{v}$, with $\left| \nabla \hat{v} \right|_{\infty} < \infty$, be a linear polynomial, and $A^{\e}(\bx):=A(\bx/\e)$ where $A \in (C^{\infty}_{Per}(Y))^{d\times d}$ is a $Y$-periodic uniformly elliptic and bounded matrix function, and assume that $v^{\e,\eta}$ solves the micro problem
\begin{equation}
\label{Micro_Problem_Div}
\text{Micro problem: }
\left\{
\begin{array}{lll}
\partial_{tt}v^{\e,\eta}(t,\bx)  - \nabla \cdot \left( A^{\e}(\bx) \nabla v^{\e,\eta}(t,\bx)  \right)= 0,  \text{ in }  I_{\tau} \times \Omega_{\bx_{0},\eta} \\
v^{\e,\eta}(0,\bx)  = \hat{v}(\bx), \quad \partial_{t}v^{\e,\eta}(0,\bx)= 0,  \quad \text{ on } \{ t=0 \} \times {\Omega_{\bx_{0},\eta}}, \\
v^{\e,\eta}(t,\bx) - \hat{v}(\bx) \text{ is } \Omega_{\bx_{0},\eta}\text{-periodic}.
\end{array}
\right.
\end{equation}
Moreover, let
$$
F_{div}  = \left( \mathcal{K}_{\tau,\eta} \ast A^{\e}(\bx) \nabla_{\bx} v^{\e,\eta}(\cdot,\cdot) \right)(0,\bx_0), \quad \bx_0 \in \Omega_{\bx_0,\eta},
$$
and $\hat{F}_{div} = A^{0}_{div} \nabla \hat{v} (\bx_0)$, where $A^{0}_{div}$ is given by \eqref{eqn_A0Div}. Then we have
$$
\left| F_{div} - \hat{F}_{div} \right| \leq C   \left(   \dfrac{\e}{\eta} \right)^{q+2}  \left|   \nabla \hat{v} \right|_{\infty},
$$
where $C$ is a constant independent of $\e$ and $\eta$ but may depend on $K,p,q$ or $A$.
\end{theorem}
\begin{proof}
A proof of this statement can be found e.g. in the proof of the main Theorem in \cite{Arjmand_Runborg_3}.  
\end{proof}

We finish the section by presenting an averaging lemma, which will also be used later in the analysis.
\begin{lemma}\label{Lem:Averaging}(Lemma $1$ in \cite{Arjmand_Runborg_1}) Let $f$ be a $1$-periodic bounded function such that $f \in L^{\infty}(Y)$ and let $K \in \mathbb{K}^{p,q}$. Then with $\bar{f}:= \int_{0}^{1} f(y) \; dy$, and $\e \leq \eta$, we have
$$
\left|  \int_{-\eta/2}^{\eta/2} K_{\eta}(x) f(x/\e) \; dx  - \bar{f}  \right|   \leq C \left( \dfrac{\e}{\eta} \right)^{q+2} \left|  f  \right|_{\infty},
$$
where $C$ does not depend on $\e, \eta$ or $f$, but may depend on $K,p,q$.
\end{lemma}

\section{Analysis}
\label{Analysis_Sec}
When the medium is isotropic and microscopically periodic, the micro problem \eqref{Intro_Micro_Problem} is simplified as
\begin{align} \label{Micro_Isotropic_HigherDimensions}
\partial_{tt} \ueeta(t,\bx) &= a^{\e}(\bx) \triangle \ueeta(t,\bx), \nonumber    \quad \text{ in   }   I_{\tau} \times \Omega_{\bx_0,\eta} \\
\ueeta(0,\bx)&= \hat{u}(\bx) , \quad \partial_t \ueeta(0,\bx) = 0,  \text{ on }   \{ t=0\} \times \Omega_{\bx_0,\eta} \\
\ueeta(t,\bx) &- \hat{u}(\bx)  \quad  \text{   is  periodic in }   \Omega_{\bx_0,\eta}  \nonumber .
 \end{align}
Here $a^{\e}(\bx) = a(\bx/\e)$, where $a$ is a $Y$-periodic coefficient, $\hat{u}(\bx)$ is a quadratic polynomial in $d$-dimensions, and $\triangle$ is the usual Laplace operator in $d$-dimensions. The main aim is to prove that the upscaled quantity given by \eqref{Intro_HMM_Flux} approximates the quantity $\hat{F}$ given by 
$$
\hat{F} = a^{0} \triangle \hat{u},
$$
where $a^{0}$ is the harmonic mean in \eqref{Eqn_HarmonicMean}. The precise statement of the main Theorem is as follows:
\begin{theorem} \label{Thm_MultiD_F}Let $\hat{u}$ be a quadratic polynomial, and assume that $A^{\e}(\bx):=a(\bx/\e) I$ where $a \in C^{\infty}(Y)$ is a $Y$-periodic, positive, and bounded coefficient, and that $u^{\e,\eta}$ solves the micro problem \eqref{Intro_Micro_Problem} in dimensions $d=1,2$ or $d=3$. Then\footnote{Note that, when the periodic coefficient $A^{\e}(\bx)$ is isotropic, i.e., $A^{\e}(\bx)  = a(\bx/\e) I$, then not all the mixed second derivatives are present in the homogenised equation 
\eqref{Effective_Eqn}; hence the operator $\nabla^2 \hat{u}$ in  \eqref{Intro_HMM_Flux} and \eqref{Definition_Fhat} is reduced to $\Delta \hat{u}$.}
\begin{equation}
\label{Ineq_Main_Estimate}
\sup_{\bx_0 \in \overline{\Omega}}\left| F(\bx_{0},\triangle \hat{u}) - \hat{F}(\bx_0,\triangle \hat{u}) \right| \leq C   \left(   \dfrac{\e}{\eta} \right)^{q+2}  \left|   \triangle \hat{u} \right|_{\infty},
\end{equation}
where $F$ and $\hat{F}$ are given by \eqref{Intro_HMM_Flux} and \eqref{Definition_Fhat}, respectively, $\eta > \e$, and $C$ is independent of $\e$ and $\eta$ but may depend on $K,p,q$ or $A$.
\end{theorem}

 Note that, in Theorem \ref{Ineq_Main_Estimate}, the choice of $\eta > \e$  is crucial in order to make the error small (by taking larger values for $q$).

The proof of this theorem in one-dimension is fairly short, while the proof in higher dimensions ($d=2$ or $d=3$) requires some additional work. Therefore, we separate the analysis, and start with proving the Theorem in one-dimension.

\begin{proof}(Proof of the Theorem \ref{Thm_MultiD_F} in one-dimension)
The micro problem \eqref{Intro_Micro_Problem}, with $\hat{u}(x) = s_0  + s_1 x + s_2 x^2$, becomes
\begin{equation*}
\left\{
\begin{array}{lll}
\partial_{tt}u^{\e,\eta}(t,x)  - a(x/\e) \partial_{xx} u^{\e,\eta}(t,x) = 0,  \\
u^{\e,\eta}(0,x)  = s_0 + s_1 x + s_2 x^2, \quad \partial_{t}u^{\e,\eta}(0,x)= 0, \\
\ueeta - \hat{u} \quad \text{ is periodic in } \Omega_{x_0,\eta}. 
\end{array}
\right.
\end{equation*}
Now let us define $v^{\e,\eta}(t,x):= \partial_x u^{\e,\eta}(t,x)$, and rewrite $F$ as 
\begin{align*}
F  &= \left( \mathcal{K}_{\tau,\eta} \ast a(\cdot / \e) \partial_{x x} \ueeta(\cdot,\cdot) \right)(0,x_{0}) \\
&= \left( \mathcal{K}_{\tau,\eta} \ast a(\cdot / \e) \partial_{x} v^{\e,\eta}(\cdot,\cdot) \right)(0,x_{0}).
\end{align*}
Next take the derivative of the micro model with respect to $x$ to see that, with $\hat{v}(x):= \partial_x \hat{u}(x) = s_1 + \frac{s_2}{2} x$,
\begin{equation*}
\left\{
\begin{array}{lll}
\partial_{tt}v^{\e,\eta}(t,x)  - \partial_x \left( a(x/\e) \partial_{x} v^{\e,\eta} \right) = 0,  \\
v^{\e,\eta}(0,x)  = \hat{v}(x), \quad \partial_{t}v^{\e,\eta}(0,x)= 0, \\
v^{\e,\eta} - \hat{v} \quad  \text{ is periodic in } \Omega_{x_0,\eta}. 
\end{array}
\right.
\end{equation*}
The last equation is  a wave equation in divergence form, and the Theorem \ref{Thm_FDiv} is applicable. In one-dimensional periodic media, the homogenised coefficients for the divergence and the nondivergence structures are the same, see Remark \ref{Rem_OneDimen_Hom}, and are given by $a^{0} = \left( \int_{0}^{1} \frac{1}{a(y)} \; dy \right)^{-1}$. Hence by the definition of $\hat{v}$, and an application of the Theorem \ref{Thm_FDiv}, it follows that

$$
 \left|  F  - a^{0} \partial_{xx} \hat{u}  \right| = \left| F  - a^{0} \partial_{x} \hat{v}   \right|  \leq  C \left(  \dfrac{\e}{\eta}  \right)^{q+2} | \partial_{xx} \hat{u}|.
$$
This completes the proof of the theorem in one-dimension.
\end{proof}
The idea behind the proof in one dimension is that a differentiation of the wave equation in non-divergence form brings the equation into the divergence form and hence earlier results relying on the symmetry of the divergence structure can be used to complete the proof. In higher dimensions a symmetric operator can be obtained by writing a system of equations for the spatial derivatives of the microscopic solutions, and a generalisation of the earlier results for the scalar equations will be needed to complete the proof. We start with an outline of the proof of the multidimensional setting. 

{\bf{Outline of the proof in multi-dimensions}:} 
\bigskip

{\bf Step 1.} We pose the micro-problem \eqref{Micro_Isotropic_HigherDimensions} over the entire $\mathbb{R}^d$. This does not cause any change in computational results and is only to simplify the analysis. The boundary conditions of the micro problem do not have any effect in the interior solution by the finite speed of propagation of waves, see  Remark \ref{Rem_FiniteSpeedOfWaves}. Therefore, the replacement of the micro problem \eqref{Micro_Isotropic_HigherDimensions} with an infinite domain counterpart, i.e., equation \eqref{Micro_Isotropic_InfiniteDomain}, is safely allowed.

{\bf Step 2.} We introduce new variables $\{ v^{\e,\eta}_i(t,\bx):= \partial_{x_i} u^{\e,\eta}(t,\bx) \}_{i=1}^{d}$, and write down explicit equations for $v_{i}^{\e,\eta}$.

{\bf Step 3.} We do the rescaling $\e \tilde{v}_i(t/\e , \bx/\e) := v^{\e,\eta}_i(t,\bx)$, and write down a system of coupled PDEs for $U(t,\by) = [\tilde{v}_1(t, \by), \tilde{v}_2(t,\by), \ldots, \tilde{v}_d(t,\by)]$. Moreover, we present the relevant function spaces, and study the regularity of a related system, which will then be used in Step $4$.

{\bf Step 4.} We introduce the local time averages $d_{\tilde{v}_i}(\by) := \left( K_{\tau} \ast \tilde{v}_i \right)(0,\by)$, and write down explicit equations for all $d_{\tilde{v}_i}$.

{\bf Step 5.} We express the upscaled quantity; i.e., $F$ given by \eqref{Intro_HMM_Flux}, in terms of the local averages $d_{\tilde{v}_i}$, and give the final estimate.

\begin{proof} (Proof of the Theorem \ref{Thm_MultiD_F} in higher dimensions)

{\bf Step 1.} We start by posing the micro problem \eqref{Micro_Isotropic_HigherDimensions} over the entire space $\mathbb{R}^d$.  The infinite domain problem reads as

\begin{equation} \label{Micro_Isotropic_InfiniteDomain}
\left\{
\begin{array}{ll}
\partial_{tt} \ueeta(t,\bx) &= a^{\e}(\bx) \triangle \ueeta(t,\bx),   \quad \text{ in   }   I_{\tau} \times \mathbb{R}^d\\
\ueeta(0,\bx)&= x_1^2 , \quad \partial_t \ueeta(0,\bx) = 0,  \text{ on }   \{ t=0\} \times \mathbb{R}^d,
\end{array}
\right.
\end{equation}

where we have assumed, without loss of generality, that $\hat{u}(\bx) = x_1^2$.  This is justified by the facts that i) $\ueeta = 0$ if the initial data is of the form $ \sum_{i=1}^{d}\sum_{j \neq i}  \alpha_{i,j} x_i^{k} x_j^{l}$ for all $k,l \in \{ 0,1 \}$ such that $k+l \leq 1$, ii) the problem is linear with respect to the initial data, iii) the upscaled quantity is symmetric with respect to the spatial coordinates $x_1$,$x_2$, $\ldots,x_d$.

{ \bf Step 2.} Let $v^{\e,\eta}_1(t,\bx) = \partial_{x_1} \left( \ueeta(t,\bx)  - x_1^2  \right)$, and $\{v^{\e,\eta}_i(t,\bx) = \partial_{x_i} \ueeta(t,\bx)\}_{i=2,3}$.  Let 
$$
L_{ij}^{\e}[ \varphi ](\bx):= \partial_{x_i} \left(  a^{\e} \partial_{x_j} \varphi \right)(\bx), \quad  i,j = 1,\ldots,d.
$$
Then taking the derivative of \eqref{Micro_Isotropic_InfiniteDomain} with respect to $x_1$, and using the relation $v_1^{\e,\eta} + 2 x_1  = \partial_{x_1} u^{\e,\eta}$, we obtain
\begin{align} \label{Eqn_w1}
\partial_{tt} v^{\e,\eta}_1(t,\bx) - L_{11}^{\e}[v^{\e,\eta}_1](t,\bx) &= L_{12}^{\e}[ v^{\e,\eta}_2](t,\bx) + L_{13}^{\e}[v^{\e,\eta}_3](t,\bx)+2 \partial_{x_1} a^{\e}(\bx), \\
v^{\e,\eta}_1(0,\bx) &= 0, \quad \partial_t v^{\e,\eta}_1(0,\bx) = 0. \nonumber
\end{align}
Similarly  the derivative with respect to $x_2$, and $x_3$ results in
\begin{align} \label{Eqn_u2}
\partial_{tt} v^{\e,\eta}_2(t,\bx) - L_{22}^{\e}[v^{\e,\eta}_2](t,\bx) &= L_{21}^{\e}[ v^{\e,\eta}_1](t,\bx)  + L_{23}^{\e}[v^{\e,\eta}_3](t,\bx)+ 2 \partial_{x_2} a^{\e}(\bx), \\
v^{\e,\eta}_2(0,\bx) &= 0, \quad \partial_t v^{\e,\eta}_2(0,\bx) = 0, \nonumber
\end{align}
and
\begin{align} \label{Eqn_u3}
\partial_{tt} v^{\e,\eta}_3(t,\bx) - L_{33}^{\e}[v^{\e,\eta}_3](t,\bx) &= L_{31}^{\e}[ v^{\e,\eta}_1](t,\bx)  + L_{32}^{\e}[v^{\e,\eta}_2](t,\bx)+ 2 \partial_{x_3} a^{\e}(\bx), \\
v^{\e,\eta}_3(0,\bx) &= 0, \quad \partial_t v^{\e,\eta}_3(0,\bx) = 0. \nonumber
\end{align}
{\bf Step 3.} We now introduce $$\{ \e \tilde{v}_i(t/\e,\bx/\e) := v^{\e,\eta}_{i}(t,\bx)\}_{i=1,2,3}.$$
Moreover, letting 
$$
U := \begin{pmatrix}
\tilde{v}_1 \\
\tilde{v}_{2} \\
\tilde{v}_3
\end{pmatrix},  \quad L := -
 \begin{pmatrix}
  \partial_{x_1} \left( a \partial_{x_1} \right) & \partial_{x_1} \left( a \partial_{x_2} \right)  & \partial_{x_1} \left( a \partial_{x_3} \right)  \\
  \partial_{x_2} \left( a \partial_{x_1} \right) & \partial_{x_2} \left( a \partial_{x_2} \right)  & \partial_{x_2} \left( a \partial_{x_3} \right)  \\
  \partial_{x_3} \left( a \partial_{x_1} \right) & \partial_{x_3} \left( a \partial_{x_2} \right)  & \partial_{x_3} \left( a \partial_{x_3} \right)
 \end{pmatrix}
,$$
and writing equations \eqref{Eqn_w1} and \eqref{Eqn_u2} in terms of the rescaled variables $U$, we arrive at\footnote{Note that, the number $2$ in the right hand side of the equation \eqref{Systemform_Eqn} can be replaced by $\triangle \hat{u}$, since  $ \triangle \hat{u}=2$. See Steps $1$ and $2$ to verify this.}
 
\begin{equation} \label{Systemform_Eqn}
\begin{array}{ll}
\partial_{tt} U(t,\by)  + L[U] = 2 \nabla a(\by), \\
U(0,\by)  = {\bf 0}, \quad \partial_t U(0,\by) ={\bf 0}.
\end{array}
\end{equation}
A few useful properties can be immediately observed from \eqref{Systemform_Eqn}. First, as the coefficient $a$ is $Y:=[0,1]^d$-periodic, it follows that the solution $U(t,\cdot)$ is also $Y$-periodic. Moreover, integrating the equation in the unit cell $Y$, we obtain
\begin{align*}
\partial_{tt} \int_Y  U(t,\by) \; d\by  + \int_Y L[U](t,\by) \; d\by   = 2 \int_Y \nabla a(\by) \; d\by.
\end{align*}
The right hand is equal to zero as the function $a$ is $Y$-periodic. Moreover, the second term in the left hand side is zero since $L[U] := - \nabla \left(  a \nabla \cdot U \right) $, and $a$ and $U$ are periodic in $Y$. We are then left with 
$$
\partial_{tt} \int_Y  U(t,\by) \; d\by  = 0.
$$
This equality, together with the zero initial data in \eqref{Systemform_Eqn}, implies that 
\begin{equation*} 
\int_Y U(t,\by) \; d\by  = 0, \quad \text{ for all }  t>0.
\end{equation*}
Let us also define the space 
\begin{equation}\label{Hdiv_Space_Def}
H_{div}(Y) := \{  U \in \left( L^{2}_{per}(Y) \right)^d, \quad  \nabla \cdot U \in L^{2}_{per}(Y)  \}.
\end{equation}
The following Lemma gives the coercivity of the operator $L$ in \eqref{Systemform_Eqn} in an appropriate function space.

\begin{lemma} \label{Lem:Coercivity}Let 
\begin{equation}\label{Def:FunctionSpaceV}
\mathcal{V} = \{f :  f \in  H_{div}(Y),  \int_{Y} f(\by) \; d\by = {\bf 0},  \text{ and }  \nabla \times f = {\bf 0} \}.
\end{equation}
Then the symmetric bilinear form $B : \mathcal{V} \times \mathcal{V} \longrightarrow \mathbb{R}$, given by
$$
B[U,V] = \int_{Y} a (\nabla \cdot U) (\nabla \cdot V) \; d\by
$$
is continuous and coercive with respect to the norm defined by
$$
\langle  U , V \rangle = \int_{Y} U(\by) \cdot V(\by)  \; d\by + \int_{Y} (\nabla \cdot U) (\nabla \cdot V) \; d\by, \quad \| U \|  = \sqrt{\langle U,U\rangle}, \quad U,V \in \mathcal{V}.
$$
\end{lemma}
\begin{proof} The continuity $B[U,V] \leq \| U \| \| V \|$ is clear. 
The coercivity follows from 
\begin{align*}
B[U,U] &= \int_Y a \left( \nabla \cdot U \right) \left( \nabla \cdot U  \right) \; d\by \geq C \| U \|^2.
\end{align*}
Note that in the last step, we used the inequality $\| U \|_{L^2(Y)}    \leq   C \|  \nabla \cdot U  \|_{L^2(Y)}$; which holds since $U$ is a curl-free field. In other words, since $\nabla \times U= 0$, it follows that $U = \nabla \Phi  + {\bf c}$, where $\Phi$ is a $Y$-periodic function, and ${\bf   c}$ is a constant vector. Furthermore,  since $U$ has zero average, it follows that ${\bf c} = {\bf  0}$. Finally, taking the divergence of $\nabla \Phi$, and applying Elliptic regularity, we obtain $\| \nabla \Phi \|_{L^2(Y)} \leq \|  \nabla \cdot U  \|_{L^2(Y)}$.
\end{proof}
We end this step by giving a regularity result for time independent equations involving high order powers of the operator $L$ in equation \eqref{Systemform_Eqn}. This result, summarised in Lemma \ref{Lem:HighOrderRegularity}, together with Lemma \ref{Lem:Coercivity} will be used later in Step $4$.
\begin{lemma} \label{Lem:HighOrderRegularity}Suppose that $U,f \in C^{\infty}_{per}(Y) \cap \mathcal{V}$, where the space $\mathcal{V}$ is defined in \eqref{Def:FunctionSpaceV}, and assume that 
$$
L^n[ U ](\by)  = f(\by), \quad \text{ in }  Y, 
$$
where $n \geq 1$ is a positive integer. Then, the following regularity result holds
$$
\|  U  \|_{H_{div}(Y)} \leq C \|  f  �\|_{L^2(Y)},
$$
where $C$ is independent of $f$ but may depend on $n$.
\end{lemma}
\begin{proof} The result with $n=1$ follows from the proof of the Lemma \ref{Lem:Coercivity}. Assume that the result is true for $n-1$, i.e., if $L^{n-1}[U] = \Phi$, then $\| U \|_{H_{div}(Y)} \leq C \| \Phi \|_{L^2(Y)}$. To prove that the result holds also for $L^{n}[U] = f$, we write
$$
L^{n}[U] = L[  \Phi  ] = f, \text{ where }   \Phi =  L^{n-1}[U].
$$
Then by the assumption from $n-1$, $ \| U \|_{H_{div}(Y)}  \leq C \|  \Phi \|_{L^2(Y)}$. The final result is obtained by observing that $\|  \Phi \|_{L^2(Y)} \leq \|  \Phi \|_{H_{div}(Y)} \leq C \|   f�\|_{L^2(Y)}$.
\end{proof}

{\bf Step 4.} Here, we start by presenting a Lemma, which gives  explicit equations for the local time averages of the solution $U$ of the equation \eqref{Systemform_Eqn}.

\begin{lemma} \label{lem:kernel2}
	Assume that $a \in C^{\infty}_{per}({\mathbb{R}^{d}})$ is $Y$-periodic, positive and bounded.
	Furthermore let $f \in \left( C^{\infty}_{per}(\mathbb{R}^{d}) \right)^d$ be a $Y$-periodic function with $\overline{f}:=\int_Y f(\by) \;d\by = 0$, $K \in \mathbb{K}^{p,q}$ with an even $q$, the operator $ L[U]  := -\nabla ( a(\by) \nabla \cdot U)$ (also given in \eqref{Systemform_Eqn}), and $U \in \mathcal{V}$ the solution of the problem
	\begin{equation} \label{Eqn_Lemma_TimeAveraging}
		\left\{ \begin{aligned}
			\partial_{tt} U(t,\by) &+ L[U](t,\by)= f(\by), && \text{ in  }   \{ t > 0\} \times \mathbb{R}^d\\
			U(0,\by) &= \partial_t U(0,\by) = 0, && \text{ on } \{ t=0 \}   \times \mathbb{R}^d. \\
					\end{aligned} \right.
	\end{equation}
	Let the local time average $d_{U}$ be defined as 
	$$
                 d_U(\by) := \int_{\mathbb{R}} K_{\tau}(t) U(\frac{t}{\e}, \by) \; dt.
	$$
	Then for $0 < \e \leq \tau$, the local time average $d_U$ satisfies
	\begin{equation} \label{eq:kernel2_help}
		L[d_U](\by) = f(\by) + \left(\frac{\e}{\tau}\right)^{q+2} R(\by),
	\end{equation}
	where $R$ is $Y$-periodic with zero average, and
	\begin{equation} \label{eq:kernel2_bound}
		\| R \|_{H_{div}(Y)} \leq C  \| f \|_{L^2(Y)}.
	\end{equation}
\end{lemma}

\begin{proof} The proof of this Lemma for scalar equations and when $L = -\nabla \cdot \left( a(\by)  \nabla \right)$ is given in \cite{Arjmand_Stohrer}. The main idea of the proof is to express the solution as an eigenfunction expansion. The proof in our setting uses precisely the same idea, but an additional regularity result is needed for time-independent systems of the form $L[W]  = f$, to be able to follow the proof in \cite{Arjmand_Stohrer}. To improve the readability and for the sake of completeness, we provide the full proof here. Since the operator $L$ is symmetric and positive, we can write the solution $U \in \mathcal{V}$ of the equation \eqref{Eqn_Lemma_TimeAveraging} in the following manner
$$
U(t,\by) = \sum_{j=1}^{\infty} u_{j}(t) \varphi_j(\by), \quad \text{ where } \varphi_j \in \left( C^{\infty}_{per}(Y) \right)^d, \text{ and } L[\varphi_j]= \lambda_j \varphi_j.
$$
By the standard theory of self-adjoint positive operators \cite{Krein_Rutman_1948}, all the eigenvalues are real and strictly positive, i.e., 
$$0 < \lambda_1 \leq \lambda_2 \leq \ldots,$$
and $\{ \varphi_j \}_{j=0}^{\infty}$ forms an orthonormal basis for $Y$-periodic functions in  $(L^2(Y))^{d}$. Plugging the expansion $U(t,\by) = \sum_{j=1}^{\infty} u_{j}(t) \varphi_j(\by)$ into the equation \eqref{Eqn_Lemma_TimeAveraging} and using the orthogonality of the eigenfunctions, we obtain
\begin{align*}
u_{j}^{\prime\prime}(t) - \lambda_j^2 u_j(t)  = f_j, \text{ where }  f(\by)  = \sum_{j=1}^{\infty}  f_j \varphi_j(\by),
\end{align*}
with homogeneous initial data $u_j(0)=u_j^{\prime}(0) = 0$. The solution of the above ODE is given explicitly as
$$
u_j(t)  = \dfrac{f_j}{\lambda_j} \left( 1 - \cos(\sqrt{\lambda_j} t) \right).
$$
Let 
$$
c_j  := \left(  \dfrac{\e}{\tau} \right)^{-q-2} \int_{\mathbb{R}} K_{\tau}(t) \cos(\dfrac{\sqrt{\lambda_j} t}{\e}) \; dt.
$$ 
Moreover, by Lemma \ref{Lem:Averaging} we get 
\begin{equation}\label{Bound:Cj}
|c_j| \leq \left(  \dfrac{\e}{\tau} \right)^{-q-2} C \left(   \dfrac{\e}{\tau \sqrt{\lambda_j}} \right)^{q+2} = C \dfrac{1}{\lambda_j^{q+2}}.
\end{equation}
Then
$$
d_U(\by)  = \sum_{j=1}^{\infty} \dfrac{f_j}{\lambda_j} \varphi_j(\by) - \left( \dfrac{\e}{\tau} \right)^{q+2} \sum_{j=1}^{\infty} \dfrac{f_j}{\lambda_j} c_j \varphi_j(\by).
$$
Hence,
\begin{align*}
L[d_{U}](\by)  &= \sum_{j=1}^{\infty} \dfrac{f_j}{\lambda_j} L[\varphi_j](\by) - \left(  \dfrac{\e}{\tau}  \right)^{q+2} L[\sum_{j=1}^{\infty}  \dfrac{f_j}{\lambda_j} c_j \varphi_j](\by)  \\
&= \sum_{j=1}^{\infty} f_j \varphi_j  + \left( \dfrac{\e}{\tau} \right)^{q+2} R(\by)  = f(\by) + \left(   \dfrac{\e}{\tau} \right)^{q+2} R(\by).
\end{align*}
Note that $d_U$ is periodic and that 
$$
R(\by) = L[ \sum_{j=1}^{\infty} \dfrac{f_j}{\lambda_j} c_j \varphi_j ](\by).
$$
We now apply the operator $L^{q/2+1}$ to $R$ and obtain
$$
L^{q/2+1}[R] = \sum_{j=1}^{\infty} f_j c_j \lambda_j^{q/2+1} \varphi_j(\by).
$$
Up to this point, the proof was precisely as in \cite{Arjmand_Stohrer}. Now, we deviate from the scheme of the proof in \cite{Arjmand_Stohrer} and instead use the regularity result given in Lemma \ref{Lem:HighOrderRegularity}, and Parseval's identity, to obtain
$$
\|  R \|_{H_{div}(Y)}^2 \leq C \|  \sum_{j=1}^{\infty}  f_j c_j \lambda_j^{q/2+1}    \varphi_j \|^2_{L^2(Y)}  = C \left|  \sum_{j=1}^{\infty} f_j^2 c_j^2 \lambda_j^{q+2} \right|.
$$
By \eqref{Bound:Cj}, it follows that
$$
\| R \|_{H_{div}(Y)}^2 \leq C \|   f \|^2_{L^2(Y)}.
$$
\end{proof}

A direct application of the Lemma \ref{lem:kernel2} to the equation \eqref{Systemform_Eqn} yields
$$
L[d_{U}](\by)  = - 2\nabla a(\by) + \alpha^{q+2}  R(\by),
$$
where $\alpha := \e/\tau$. From here, we can see that
\begin{equation} \label{Eqn_d1}
-\partial_{y_1}   \left( a(\by) \partial_{y_1} d_{\tilde{v}_1}(\by)   + a(\by) \partial_{y_2} d_{\tilde{v}_2}(\by)   +  a(\by) \partial_{y_3} d_{\tilde{v}_3}(\by)+ 2 a(\by) \right) = \alpha^{q+2} R_1(\by),
\end{equation}
\begin{equation}\label{Eqn_d2}
-\partial_{y_2}   \left( a(\by) \partial_{y_1} d_{\tilde{v}_1}(\by)   + a(\by) \partial_{y_2} d_{\tilde{v}_2}(\by) +  a(\by) \partial_{y_3} d_{\tilde{v}_3}(\by)+ 2 a(\by) \right) =\alpha^{q+2}  R_2(\by),
\end{equation}
and
\begin{equation}\label{Eqn_d3}
-\partial_{y_3}   \left( a(\by) \partial_{y_1} d_{\tilde{v}_1}(\by)   + a(\by) \partial_{y_2} d_{\tilde{v}_2}(\by) +  a(\by) \partial_{y_3} d_{\tilde{v}_3}(\by)+ 2 a(\by) \right) =\alpha^{q+2}  R_3(\by).
\end{equation}
\begin{remark} Note that by applying the Lemma  \ref{lem:kernel2} to \eqref{Systemform_Eqn}, we can bound the remainders $R_{i}$ as 
\begin{equation}\label{Ineq:R_Bound}
\| R_i \|_{L^2(Y)} \leq \| R \|_{H_{div}(Y)} \leq C_0 \|  \underbrace{2}_{\triangle \hat{u}}    \nabla a \|_{L^2(Y)} \leq C_1 \left| \triangle \hat{u} \right|.
\end{equation}
\end{remark}

{\bf Step 5.} In this step, we first rewrite the upscaled quantity ($F$ in \eqref{Intro_HMM_Flux}) in terms of $d_{\tilde{v}_i}(\by)$. Following, the precise notations used in Steps $2,3$, and $4$, we write 
\begin{align} \label{Eqn_HMM_Flux_ReFormulated}
F(\bx_0)  &= \left( \mathcal{K}_{\tau,\eta} \ast a^{\e}(\cdot) \triangle \ueeta(\cdot,\cdot) \right)(0,\bx_0) \nonumber\\
&=  \left( \mathcal{K}_{\tau,\eta} \ast a(\cdot/\e) \left(  \partial_{x_1} \partial_{x_1} \ueeta(\cdot,\cdot)   +   \partial_{x_2} \partial_{x_2} \ueeta(\cdot,\cdot) + \partial_{x_3} \partial_{x_3} \ueeta(\cdot,\cdot)  \right) \right)(0,\bx_0) \nonumber \\
&=  \left( \mathcal{K}_{\tau,\eta} \ast a(\cdot/\e) \left(  \partial_{x_1} v^{\e,\eta}_1(\cdot,\cdot)  +   \partial_{x_2}v^{\e,\eta}_2(\cdot,\cdot) + \partial_{x_3}v^{\e,\eta}_3(\cdot,\cdot)+ 2  \right) \right)(0,\bx_0) \nonumber \\
&= \left( \mathcal{K}_{\tau,\eta} \ast a(\cdot/\e) \left(  \partial_{y_1} \tilde{v}^{\e,\eta}_1(\cdot/\e,\cdot/\e)  +   \partial_{y_2}\tilde{v}^{\e,\eta}_2(\cdot/\e,\cdot/\e) + \partial_{y_3}\tilde{v}^{\e,\eta}_3(\cdot/\e,\cdot/\e)+  2  \right) \right)(0,\bx_0) \nonumber \\
&= \left( K_{\eta} \ast \left(   a(\cdot/\e) \left(  \partial_{y_1} d_{\tilde{v}_1}(\cdot/\e)  +   \partial_{y_2}d_{\tilde{v}_2}(\cdot/\e) + \partial_{y_3}d_{\tilde{v}_3}(\cdot/\e)+ 2  \right) \right) \right)(\bx_0)  \nonumber \\
&=: \int_{\Omega_{\eta,\bx_0}} K_{\eta}(\bx - \bx_0)   a(\bx/\e)  \left(  \partial_{y_1} d_{\tilde{v}_1}(\bx/\e)  +   \partial_{y_2}d_{\tilde{v}_2}(\bx/\e) + \partial_{y_3}d_{\tilde{v}_3}(\bx/\e) + 2  \right) \; d\bx
\end{align}
On the other hand, integrating \eqref{Eqn_d1} with respect to $y_1$, the equation \eqref{Eqn_d2} with respect to $y_2$, and  \eqref{Eqn_d3} with respect to $y_3$,  we obtain
\begin{align*}
a(\by) \left( \partial_{y_1} d_{\tilde{v}_1}(\by)   + \partial_{y_2} d_{\tilde{v}_2}(\by)   + \partial_{y_3} d_{\tilde{v}_3}(\by)+2  \right) &= C_1(y_2,y_3)  + \alpha^{q+2} \int_{0}^{y_1} R_1(z_1,y_2,y_3) \; dz_1 \\ &= C_2(y_1,y_3) + \alpha^{q+2} \int_{0}^{y_2}R_2(y_1,z_2,y_3) \; dz_2  \\�
&= C_3(y_1,y_2)  + \alpha^{q+2} \int_{0}^{y_3}R_3(y_1,y_2,z_3) \; dz_3.
\end{align*}
Equating the equal powers in the last equality, we readily observe that $C_1(y_2,y_3) = C_2(y_1,y_3)=C_3(y_1,y_2)$, which implies that $C_1 = C_2 =C_3 = C$; a constant independent of $\by$. The constant $C$ can be found by dividing the first equation with $a(\by)$ and integrating the resulting equation over the unit cube $Y$. This yields
\begin{eqnarray*}
\int_Y \partial_{y_1} d_{\tilde{v}_1}(\by)   +  \partial_{y_2} d_{\tilde{v}_2}(\by)   + \partial_{y_3} d_{\tilde{v}_3}(\by) +  2 \; d\by   = C  \int_Y \dfrac{1}{a(\by)} \; d\by +  
\\ + \alpha^{q+2} \int_Y \dfrac{1}{a(\by)} \int_{0}^{y_1} R_1(z_1,y_2,y_3) \; dz_1 d\by. 
\end{eqnarray*}
Multiplying both sides by $ a^0 = \left(\int_{Y} \frac{1}{a(\by)} \; d\by \right)^{-1}$, we obtain
\begin{equation*}
2 a^{0} =  C + \alpha^{q+2} a^{0}  \int_Y \dfrac{1}{a(\by)} \int_{0}^{y_1} R_1(z_1,y_2,y_3) \; dz_1 d\by.
\end{equation*}
Therefore,
\begin{eqnarray*}
a(\by) \left( \partial_{y_1} d_{\tilde{v}_1}(\by)   +  \partial_{y_2} d_{\tilde{v}_2}(\by)   + \partial_{y_3} d_{\tilde{v}_3}(\by) + 2 \right) = \\
= 2 a^0 + \alpha^{q+2} \int_{0}^{y_1} R_1(z_1,y_2,y_3) \; dz_1 - \alpha^{q+2} a^{0} \int_Y \dfrac{1}{a(\by)} \int_{0}^{y_1} R_1(z_1,y_2,y_3) \; dz_1 d\by.
\end{eqnarray*}
Let $g(\by):= a(\by) \left( \partial_{y_1} d_{\tilde{v}_1}(\by)   +  \partial_{y_2} d_{\tilde{v}_2}(\by)   + \partial_{y_3} d_{\tilde{v}_3}(\by) + 2 \right)$. Clearly $g$ is $Y$-periodic, and its average stays very close to $2 a^{0}$, i.e., 
\begin{align}
\bar{g} :=& \int_Y g(\by) \; d\by \nonumber\\
=&  2 a^0 - \alpha^{q+2} a^{0} \int_Y \left( \dfrac{1}{a(\by)}  - \dfrac{1}{a^0}\right) \int_{0}^{y_1} R_1(z_1,y_2,y_3) \; dz_1 d\by.  \label{Eqn_gbar}
\end{align}
Now, observe that the last integral in \eqref{Eqn_HMM_Flux_ReFormulated} can be written in terms of $g$ as follows: 
\begin{equation}\label{Eqn_Fg}
F(\bx_0,\Delta \hat{u}) = \int_{\Omega_{\eta,\bx_0}} K_{\eta}(\bx-\bx_0) g(\bx/\e) \; d\bx.
\end{equation}
By the equations \eqref{Eqn_gbar} and \eqref{Eqn_Fg}, and using the estimate \eqref{Ineq:R_Bound}, it follows that
\begin{align*}
\left|  F(\bx_0,\Delta \hat{u})  - \hat{F}(\bx_0,\Delta \hat{u})  \right| &:= \left|   F(\bx_0,\Delta \hat{u})  - 2 a^{0}  \right| \\
& \hspace{-2cm}=  \left|   F(\bx_0,\Delta \hat{u})  - \bar{g} + \alpha^{q+2} a^{0} \int_Y \left( \dfrac{1}{a(\by)}  - \dfrac{1}{a^0}\right) \int_{0}^{y_1} R_1(z_1,y_2,y_3) \; dz_1 d\by \right| \\
&\leq \left|   F(\bx_0,\Delta \hat{u})  - \bar{g} \right| + C_0 \alpha^{q+2}  \int_Y  \left|  R_1(\by) \right|  \; d\by   \\
&\leq \underbrace{\left|  \int_{\Omega_{\eta,\bx_0}} K_{\eta}(\bx-\bx_0) g(\bx/\e) \; d\bx  - \bar{g}   \right|}_{\leq C_1 \alpha^{q+2} \| g \|_{L^{\infty}(Y)}  \text{ by Lemma }  \ref{Lem:Averaging}} + C_0 \alpha^{q+2} \underbrace{\| R�\|_{H_{div}(Y)}}_{\leq C \left|   \triangle \hat{u} \right|, \text{ by \eqref{Ineq:R_Bound}}} \\
&\leq C_2 \alpha^{q+2}  \| R�\|_{H_{div}(Y)}  \leq C_2 \alpha^{q+2} |\triangle \hat{u}|.
\end{align*}
This completes the proof of the Theorem \ref{Thm_MultiD_F}.
\end{proof}

\section{Estimates for the full numerical solution}
\label{FullyDiscrete_Sec}
Simulations in this paper use a Leap-Frog discretization for the micro- and the macromodel. Therefore, we start this section by presenting standard results for the stability estimates for the Leap-Frog scheme, which will then be used in Subsection \ref{Subsec_ConvergenceAnalysis} to give a bound for the error between the solution of the equation free approach (EFA) and the homogenised solution.

\subsection{Difference schemes for the wave equation}
Let $\mathcal{V}_H$ be a finite dimensional Hilbert space consisting of real-valued functions defined on a mesh $\Omega_H$, given by,

$$
\Omega_{H} :=\{   \bx_I=(i_1 H, i_2 H, \ldots, i_d H), i_j = 1,\ldots, N_x-1, j=1,\ldots,d,  \text{and } (N_x  - 1) H  = L  \}.
$$
The space $\mathcal{V}_H$ is equipped with the inner product and norm
$$
\langle y,v \rangle  = \sum_{i_1 = 1}^{N_x - 1}  \ldots \sum_{i_d = 1}^{N_x - 1}  y_{i_1,\ldots,i_d} v_{i_1,\ldots,i_d} H^d, \quad \left\|   y    \right\| = \sqrt{\langle  y,y    \rangle},   \text{ where }  y,v \in \mathcal{V}_H.
$$
For an operator $B: \mathcal{V}_{H} \to \mathcal{V}_H$, with $B = B^{*}$, we define the weighted norm 
$$
\left\|   y  \right\|_{B}  := \sqrt{\langle  B y, y   \rangle}. 
$$
Let $\{ y^{n}   \}_{n \geq 1}$ be a sequence of grid functions, with $y^{n} \in \mathcal{V}_H$.  We denote the standard 
second order difference operator by
$$
D_t^{2}  y^{n} := \dfrac{y^{n+1} - 2 y^{n}  + y^{n-1}}{\triangle t^2}.
$$
We first present a result from \cite{Samarskii_Book} to study the stability of the operator equation
\begin{equation} \label{Eqn_Samarskii_Difference}
\begin{array}{lll}
E_H D_t^2  y^{n} +  L_{H} y^{n}  = \varphi^{n}, \\
y^{0}  = g, \quad y^{1}  = z,
\end{array}
\end{equation}
where $E_H$, and $L_{H}$ are two given finite-dimensional operators, and $\varphi^{n}, g, z$ belong to $\mathcal{V}_H$.

\begin{theorem} [See \cite{Samarskii_Book}] \label{Estimate_Difference_Scheme}
Assume that
$$
L_H = L_H^{*}, \quad \text{ and }\quad  E_H=E_H^{*} > \dfrac{\tau^2}{4} L_{H}. 
$$
Let $y^{n}$ solve the difference equation \eqref{Eqn_Samarskii_Difference}. Then
$$
\left\|  y^{n+1}  \right\|_{L_{H}}     \leq M \left(   \left\|  y^{0}   \right\|_{L_{H}}  + \left\|   \dfrac{y^1 - y^{0}}{\delta t}   \right\|_{E_H} + \sum_{k=1}^{n} \triangle t  \left\| \varphi^{k}   \right\|_{E_H^{-1}}   \right),
$$
where $M$ is a constant independent of $\varphi$, and $t$.
\end{theorem}
Now we apply the Theorem \ref{Estimate_Difference_Scheme} for a finite-dimensional approximation of the wave equation

\begin{equation} \label{Eqn:Wave_Stability}
\left\{
\begin{array}{lll}
\partial_{tt} u(t,\bx)  = a(\bx) \triangle u(t,\bx) + f(t,\bx),   \quad \text{ in }  (0,T] \times \Omega, \\
u(0,\bx)  = g(\bx), \quad \partial_t u(0,\bx) = h(\bx), \\  
u(t,\bx)  = 0 \quad \text{ on }  \partial\Omega,
\end{array}
\right.
\end{equation}
where $a$ is a bounded, positive  wave speed, such that $c_1> a>c_0>0$, and $\Omega:=[0,L]^d$. The equation \eqref{Eqn:Wave_Stability} is discretised using the Leap-Frog scheme    
\begin{equation}\label{Eqn:LeapFrogScheme}
\left\{
\begin{array}{lll}
D_t^2 u^{n}_I = a_I \displaystyle \triangle_H u^{n}_I   + f_I^{n}, \text{ in }  \mathcal{T}_{\triangle t} \times  \Omega_{H}\\
u^{0}_I  = g_I, \quad u^{1}_I  = u^{0}_I + \triangle t  \,  h_I + \dfrac{\triangle t^2}{2} \left(  a \triangle_H u^{0}_I  + f^{0}_I \right), \\
u^{0}_I  = 0, \quad \text{ on } \bx_I \in \partial \Omega_{H},
\end{array}
\right.
\end{equation}
where $u^{n}_{I}$ approximates $u(t_n,\bx_{I})$ (the solution of \eqref{Eqn:Wave_Stability}), $I = (i_1,i_2,\ldots,i_d)$ is a multi-index, and the operator $\triangle_H$ is defined as
$$
\triangle_H u^{n}_{I} :=\sum_{j=1}^{d} \dfrac{u^{n}_{I + e_j}  - 2 u^{n}_{I}   + u^{n}_{I-e_j} }{H^2},
$$
where $e_j \in \mathbb{R}^d$ denotes the canonical basis vector in $j$-th direction. Moreover $\mathcal{T}_{\triangle t} $ is a discretisation of the time interval, given by
$$
\mathcal{T}_{\triangle t} :=\{   (t_n=n \triangle t), n = 0,\ldots, N_t-1, \text{ and } (N_t  - 1) \triangle t  = T  \}.
$$

\begin{corollary}\label{Cor_Stability_Leap_Frog}
Suppose that $u^{n}$ solves the difference scheme \eqref{Eqn:LeapFrogScheme}, and that the assumption 
$$
\triangle t < \dfrac{2 H}{d \sqrt{c_2}}
$$
holds, and let $\varphi = f/a$. Moreover, assume that $L_H := - \triangle_H$, and $I_H$ is the identity operator. Then, the stability estimate
$$
\left\| u^{n+1}  \right\|_{L_{H}}     \leq M \left(   \left\|  u^{0}   \right\|_{L_H}  + \left\|   \dfrac{u^1 - u^{0}}{\triangle t}   \right\|_{\frac{1}{a} I_H} + \sum_{k=1}^{n} \triangle t  \left\| \varphi^{k}   \right\|_{a I_H}   \right)
$$
holds. 
\end{corollary}

\subsection{Convergence analysis}
\label{Subsec_ConvergenceAnalysis}

The aim of this section is to give an outline for the error bound for the difference between the solution of the equation free approach (EFA) \eqref{Macro_Solver_Eqn}, and the solution of the homogenized PDE \eqref{Effective_Eqn}. For the analysis, we introduce 

\begin{equation*}
u^{n}_I := \text{ An approximation to the homogenized solution } u^{0}(t_n,\bx_I),
\end{equation*}

\begin{equation*}
U^{n}_I := \text{ Solution of the EFA \eqref{Macro_Solver_Eqn} when the micro-problem is solved exactly},
\end{equation*}

\begin{equation*}
\tilde{U}^{n}_I := \text{ Solution of the EFA  \eqref{Macro_Solver_Eqn} when the micro-problem is solved numerically}.
\end{equation*}
The discrete homogenized solution satisfies
\begin{equation} \label{Eqn:DiscreteHomogenized}
\left\{
\begin{array}{lll}
D_t^2 u^{n}_{I} = \hat{F}^{n}_{I}(\bx_I , \triangle_H u^{n}_{I})  + f^{n}_I, \\
u^{0}_{I}  = g_{I}, \quad u^{1}_I  = g_{I} + \triangle t h_{I}  + \dfrac{\triangle t^2}{2} \left(   \hat{F}^{0}_I(\bx_I , \triangle_H u^{0}_I)     + f^{0}_I  \right),
\end{array}
\right.
\end{equation}
where $\hat{F}_I(\bx_I.\triangle_H u^{n}_{I}) = a^{0}  \triangle_H u^{n}_{I} := a^{0} \displaystyle \sum_{i=1}^{d} \dfrac{u^{n}_{I + e_i}  - 2 u^{n}_{I}   + u^{n}_{I-e_i} }{H^2}$.  The semidiscrete EFA solution $U^{n}_{I}$ solves
\begin{equation} \label{Eqn:SemiDiscreteEFA}
\left\{
\begin{array}{lll}
 D_t^2 U^{n}_{I} = F^{n}_{I}(\bx_I , \triangle_H U^{n}_{I})  + f^{n}_I, \\
U^{0}_{I}  = g_{I}, \quad U^{1}_I  = g_{I} + \triangle t h_{I}  + \dfrac{\triangle t^2}{2} \left(   F^{0}_I(\bx_I , \triangle_H U^{0}_I)     + f^{0}_I  \right),
\end{array}
\right.
\end{equation}
where $F^{n}_I(\bx_I,\triangle_H U^{n}_I ) = \hat{F}^{n}_I(\bx_I,\triangle_H U^{n}_I )  + \delta_{I}^{n}(\triangle_H U^{n}_I)$, and 

$$\delta_I^{n}(\triangle_H U^{n}_I) \leq C \left( \dfrac{\e}{\eta} \right)^{q+2} \left| \triangle_H U^{n}_I \right|$$
is the upscaling error, which was estimated by the Theorem \ref{Thm_MultiD_F}. The fully discrete EFA solution $\tilde{U}^{n}_I$ satisfies

\begin{equation} \label{Eqn:FullyDiscreteEFA}
\left\{
\begin{array}{lll}
D_t^2 \tilde{U}^{n}_{I}  = \tilde{F}^{n}_{I}(\bx_I , \triangle_H \tilde{U}^{n}_{I})  + f^{n}_I, \\
\tilde{U}^{0}_{I}  = g_{I}, \quad \tilde{U}^{1}_I  = g_{I} + \triangle t h_{I}  + \dfrac{\triangle t^2}{2} \left(   \tilde{F}^{0}_I(\bx_I , \triangle_H \tilde{U}^{0}_I)     + f^{0}_I  \right).
\end{array}
\right.
\end{equation}

We assume that the equations \eqref{Eqn:DiscreteHomogenized},\eqref{Eqn:SemiDiscreteEFA}, and \eqref{Eqn:FullyDiscreteEFA} are equipped with homogeneous Dirichlet boundary conditions. We are now interested in estimating the difference between the fully discrete EFA solution $\tilde{U}^{n}_I$ and the exact homogenized solution $u^{0}(t_n, \bx_I)$. For this we write

$$
\left\|    u^{0}(t_n,\cdot)   -  \tilde{U}_I^{n} \right\|   \leq \underbrace{\left\|    u^{0}(t_n,\cdot)   -  u_I^{n} \right\|}_{E_{macro}}  + \underbrace{\left\|    u_I^{n}  -  U_I^{n} \right\|}_{E_{upscaling}}   + \underbrace{\left\|    U_I^{n}   -  \tilde{U}_I^{n} \right\|}_{E_{micro}}. 
$$

The first term in the right hand side is the discretization error in the macro level, the middle term is the upscaling error due to the two way coupling between the micro- and the macroscale quantities, and the last term is the discretization error in the microscopic simulations. We will now proceed with presenting error estimates for the upscaling and the discretization errors.

 \subsubsection{The upscaling error}

Let us define $e_I^{n} := U_I^{n} - u^{n}_I$. By definition, $e^{n}_I$ will then satisfy
\begin{equation}
\begin{array}{lll}
D_t^2 e^{n}_I  = \hat{F}^{n}_I(\bx_I, \triangle_H e^{n}_I) + \delta^{n}_{I}(\triangle_H U^{n}_{I}) \\
e^{0}_I  = 0, \quad e^{1}_{I}  = \dfrac{\triangle t^2}{2}  \delta^{0}_I(\triangle_H U^{0}_I).
\end{array}
\end{equation}
Now using the Theorem \ref{Estimate_Difference_Scheme}, with $L_H = - \triangle_H $, we obtain
\begin{equation}\label{Temp_Estimate}
\|   e^{n+1}_I \|_{L_H}   \leq M \left(  \dfrac{\triangle t}{2} \|   \delta^{0}_I(\triangle_H U^{0}_I)  \|_{a^{-1} I_H}  +    \sum_{k=1}^{N}   \triangle t   \| a^{-1} \delta^{n}_{I}(\triangle_H U^{0}_I) \|_{aI_H} \right).
\end{equation}
Next since $|\delta^{n}_{I}(\nabla_H U^{n}_I)| \leq C_{I}^{n} \left(   \frac{\e}{\eta} \right)^{q+2}   \left|  \triangle_H U^n_{I}   \right|$, an estimate
for the norm $\|  \triangle_H U^{n}_I �\|$ is needed. To do this, we first rewrite \eqref{Eqn:SemiDiscreteEFA} (with $L_{H}:= -\triangle_H$) as  

\begin{equation} \label{Eqn:SemiDiscreteEFA_Reform}
\left\{
\begin{array}{lll}
D_t^2 U^{n}_{I}  +\tilde{a}_I^{0} L_H  U^{n}_{I}  = f^{n}_I, \\
U^{0}_{I}  = g_{I}, \quad U^{1}_I  = g_{I} + \triangle t h_{I}  + \dfrac{\triangle t^2}{2} \left(   -\tilde{a}^{0}_I L_H  U^{0}_I)     + f^{0}_I  \right),
\end{array}
\right.
\end{equation}
where $ \left| a^{0}  - \tilde{a}^{0} \right|   \leq C \left( \frac{\e}{\eta} \right)^{q+2}$. We define also the operator $\tilde{L}_H := L_H^{1/2} \tilde{a}^{0}_I  L_H^{1/2}$, and apply the operator $L_H^{1/2}$ to \eqref{Eqn:SemiDiscreteEFA_Reform}. Let us denote $W^{n}_I = L_H^{1/2} U^{n}_I$, then

\begin{equation} \label{Eqn:SemiDiscreteEFA_Reform_Multiplied}
\left\{
\begin{array}{lll}
D_t^2 W^{n}_{I}  +  \tilde{L}_H  W^{n}_{I}  = L_{H}^{1/2} f^{n}_I, \\
W^{0}_{I}  = L_H^{1/2} g_{I}, \quad W^{1}_I  =L_{H}^{1/2} g_{I} + \triangle t L_H^{1/2} h_{I}  + \dfrac{\triangle t^2}{2} \left(   -\tilde{L}_H  U^{0}_I     + L_H^{1/2} f^{0}_I  \right),
\end{array}
\right.
\end{equation}

By definition of $W^{n}_I$, the Corollary \eqref{Cor_Stability_Leap_Frog}, and assuming that $\left( \frac{\e}{\eta} \right)^{q+2}$ is sufficiently small, we obtain 
\begin{align*}
\left\|   L_H U^{n}_I   \right\|   &\leq C \sqrt{  \langle  \tilde{a}^{0}_I L_H U^{n}_I  ,   L_H U^{n}_I  \rangle   } 
= \sqrt{ \langle   \tilde{L}_H   L_H^{1/2}  U^{n}_I   ,   L_H^{1/2} U^{n}_I \rangle } = \sqrt{ \langle    \tilde{L}_H W^{n}_I   ,  W^{n}_I  \rangle} \\  & \hspace{-1cm} \leq C \left(  \left\|  L_H^{1/2}  g_H \right\|_{\tilde{L}_H}     +   \left\|   L_H^{1/2} h_I + \dfrac{\triangle t}{2}  \left(  \tilde{L}_H  g_I    + L_H^{1/2} f^{0}_I \right)   \right\|_{\frac{1}{\tilde{a}^{0}}   I_H } + \sum_{k=1}^{n}  \triangle t \left\| \frac{1}{\tilde{a}^{0}} L_H^{1/2} f^{k}_I  \right\|_{\tilde{a}^{0} I_H}  \right).
\end{align*}
It follows that $ \left\|   L_H U^{n}_I   \right\|   \leq \tilde{C} $, if $g \in C^{2}(\overline{\Omega})$, and $h,f  \in C^{1}(\overline{\Omega})$. Using this result in the inequality \eqref{Temp_Estimate}, we obtain
$$
\left\| e^{n+1}_I \right\|  \leq C_T   \left(  \dfrac{\e}{\eta} \right)^{q+2}.
$$
\subsubsection{The macro and the micro errors}
Let $z_I^{n}:= u^{0}(t_n,\bx_I) - u^{n}_{I}$. Applying the operators $D_t^2$ and $\triangle_H$ to $z_I^{n}$, and using the fact that 
$$
\partial_{tt} u^{0}(t_n,\bx_I) = D_t^2 u^{0}(t_n,\bx_I) + C_{1,I}^{n} \triangle t^2, \text{ and }  \triangle u^{0}(t_n,\bx_I)  = \triangle_H u^{0}(t_n,\bx_I)  + C_{2,I}^{n} H^2,
$$
where $| C_{1,I}^{n} |\leq  \displaystyle \sup_{t \in [0,T],\bx \in \Omega} \partial_{ttt} u^{0}(t,\bx)$, and $| C_{2,I}^{n} |\leq  \displaystyle \sup_{t \in [0,T],\bx \in \Omega} \partial_{x_i x_j x_k} u^{0}(t,\bx)$, we obtain
\begin{equation}
\begin{array}{ll}
D_t^2 z_{I}^{n}  = a^{0}  \triangle_H z^{n}_{I} + C_{I}^{n} \left(    \triangle t^{2}  + H^2  \right), \\
z^{0}_{I}  = 0, \quad z^{1}_{I}  = \dfrac{\triangle t^2}{2} C_{I}^{0} H^2.
\end{array}
\end{equation}
Here $|C^n_{I} |  \leq \max \{|C_{1,I}^{n}|,|C_{2,I}^{n}|\}$. A direct application of the Corollary \ref{Cor_Stability_Leap_Frog} gives 
$$
\| z^{n}_I \|:= \|   u^{0}(t_n,\bx_I)    -    u^{n}_{I} \| \leq C \left(   \triangle t^2 +  H^2 \right).
$$
For the micro error, we argue similarly and find that $\left\|    U_I^{n}   -  \tilde{U}_I^{n} \right\|   \leq C \left(   \delta t^2 + \delta x^2 \right)$. Note that the stability estimates in this section are established for the Leap frog scheme; which is an explicit method. For stability estimates for other types of numerical schemes, we refer the reader to standards books on finite differences, such as \cite{Gustafsson_Kreiss_Oliger,Samarskii_Book}. For stability estimates for higher order implicit methods for the second order hyperbolic equations, see also \cite{Ashyralyev_Sobolevskii_Book}. 
\section{Numerical results}
We subdivide this section into three parts. First, in subsection \ref{Sec:UpscalingError}, the upscaling error from Theorem \ref{Thm_MultiD_F} is illustrated.  We then, in subsection \ref{Sec:OneDSoln}, compare the solution of the equation free approach in one-dimensional periodic and \emph{almost-periodic} media. Finally, in subsection \ref{Sec:TwoDSoln}, we compare our numerical solutions for a two-dimensional periodic setting.

\subsection{Upscaling error}
\label{Sec:UpscalingError}
\begin{figure}[ht] \label{Fig:UpscalingError}
	\centering
	\includegraphics[width=0.49\textwidth]{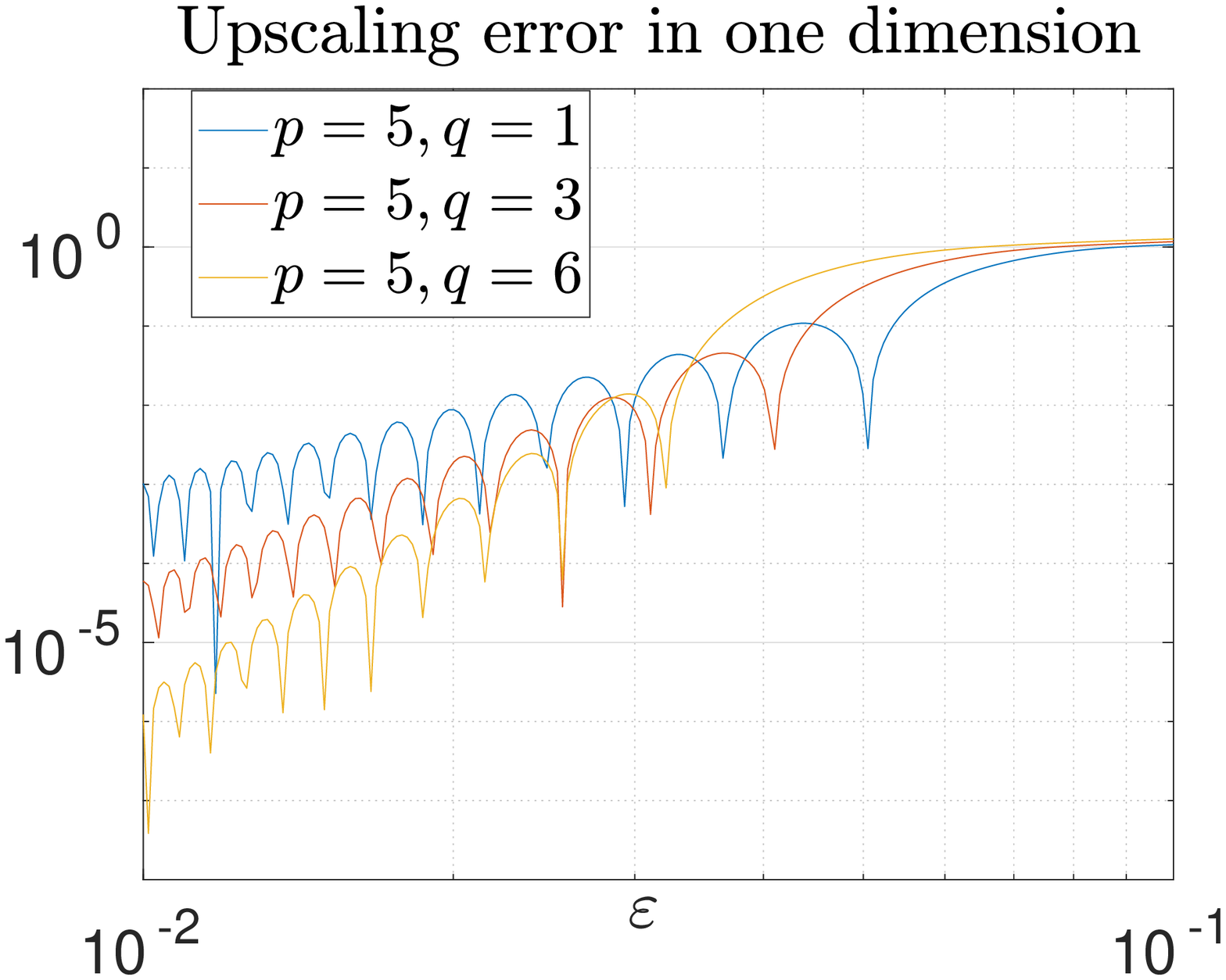}
		\includegraphics[width=0.49\textwidth]{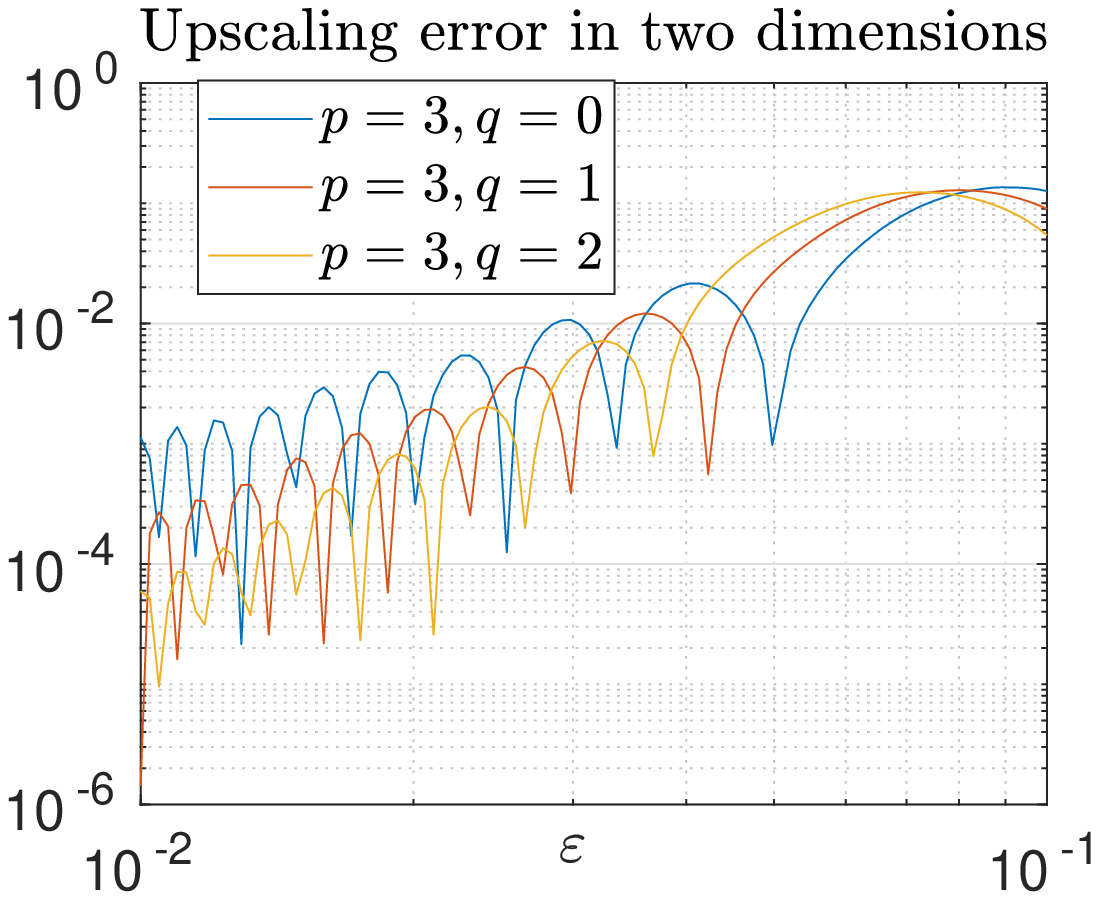}
	\caption[Upscaling errors]
	{Upscaling errors in one and two dimensional periodic media. }
\end{figure}

Here, the upscaling error in Theorem \ref{Thm_MultiD_F} is tested for periodic problems in one and two dimensions. In one-dimension, we consider the micro problem \eqref{Intro_Micro_Problem}, with an initial data $\hat{u}(x) = x^2$. The coefficient $A^{\e}(x)$ is taken to be
$$
A^{\e}(x)  = 1.1  + \sin(2\pi x/\e).
$$
In this case, the exact homogenized coefficient reads as
$$
A^{0}  = \left(  \int_0^{1}   \left(  1.1  + \sin(2 \pi y)  \right)^{-1}    \; dy  \right)^{-1} = \sqrt{1.1^2  - 1},
$$ 
and therefore, the exact upscaled quantity becomes $\hat{F}  := A^{0} \partial_{xx} \hat{u}  = 2 \sqrt{1.1^2  - 1}$. The left plot in the Figure \ref{Fig:UpscalingError}  shows the upscaling error, i.e., $| F  - \hat{F}|$, where $F$ (see \eqref{Intro_HMM_Flux}) is the upscaled quantity in the equation free approach. In this simulation, the size of the averaging box is chosen to be $\eta = \tau  = 0.1$, and the upscaling error is plotted against $\e$, for averaging kernels with different regularities. Higher values for $q$ implies better regularity properties for the kernel, and the figure shows the precise convergence rate $O( (\e/\eta)^{q+2} )$, which verifies the result of the Theorem \ref{Thm_MultiD_F}. In the right plot of the Figure \ref{Fig:UpscalingError}, we consider the micro-problem \eqref{Intro_Micro_Problem} with a two-dimensional  material coefficient
\begin{equation}\label{NumRes:2D_Coefficient}
A^{\e}(\bx)  = \left( 1.1  + \cos(2\pi x_1/\e) \sin(2\pi x_2/\e) + e^{\cos(2\pi x_1/\e) + \sin(2\pi x_2/\e)} \right)^{-1} I,
\end{equation}
where $I$ is the $2\times 2$ identity matrix.
The micro-problem \eqref{Intro_Micro_Problem} is equipped with the initial data $\hat{u}(\bx) =  x_1^2$. In this case, the exact homogenized coefficient is approximated by $10$ digits of accuracy as follows 
$$
A^{0}  = 0.3699698702 I.
$$

Therefore, the exact upscaled quantity becomes $\hat{F} = A^{0} \partial_{x_1 x_1} \hat{u} = 2 A^{0}$. Moreover, we have used $\tau = \eta = 0.1$ in the simulations. Similar to the one-dimensional case, precise convergence rates of the Theorem \eqref{Thm_MultiD_F} are observed in the simulations.

\subsection{Solution in one dimension}
\label{Sec:OneDSoln}

\begin{figure}[h] 
	\centering
	\includegraphics[width=0.49\textwidth]{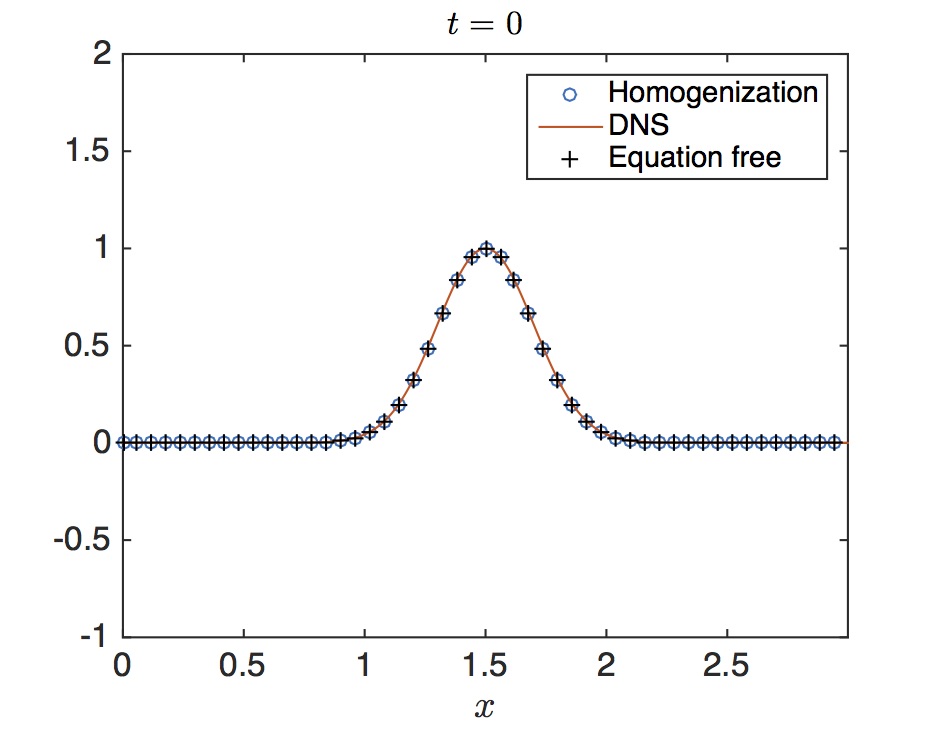}
		\includegraphics[width=0.49\textwidth]{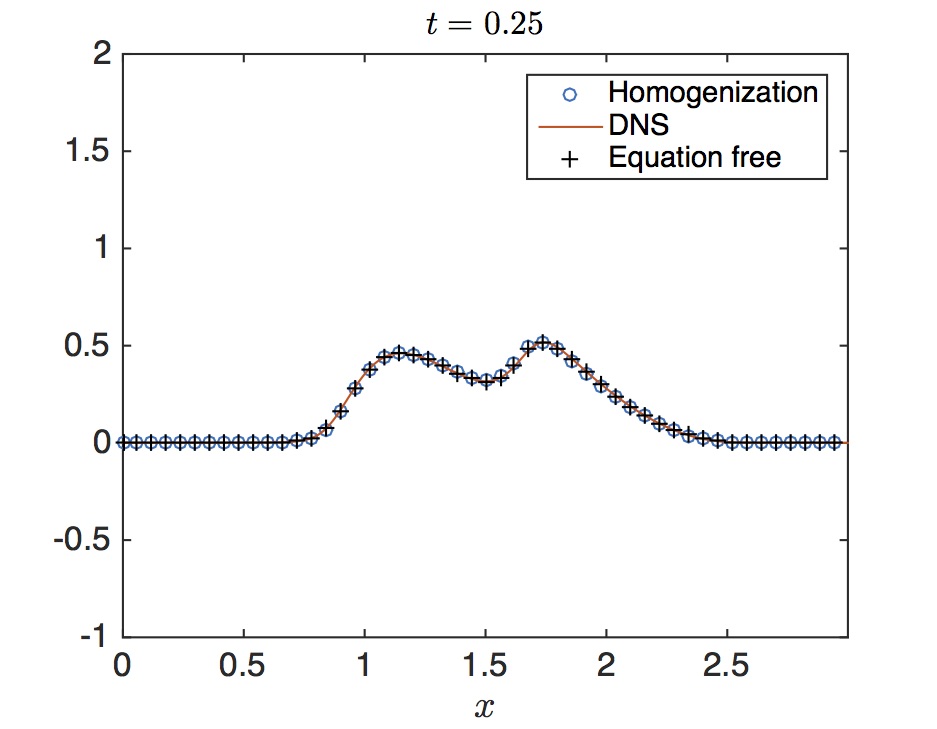}
	\includegraphics[width=0.49\textwidth]{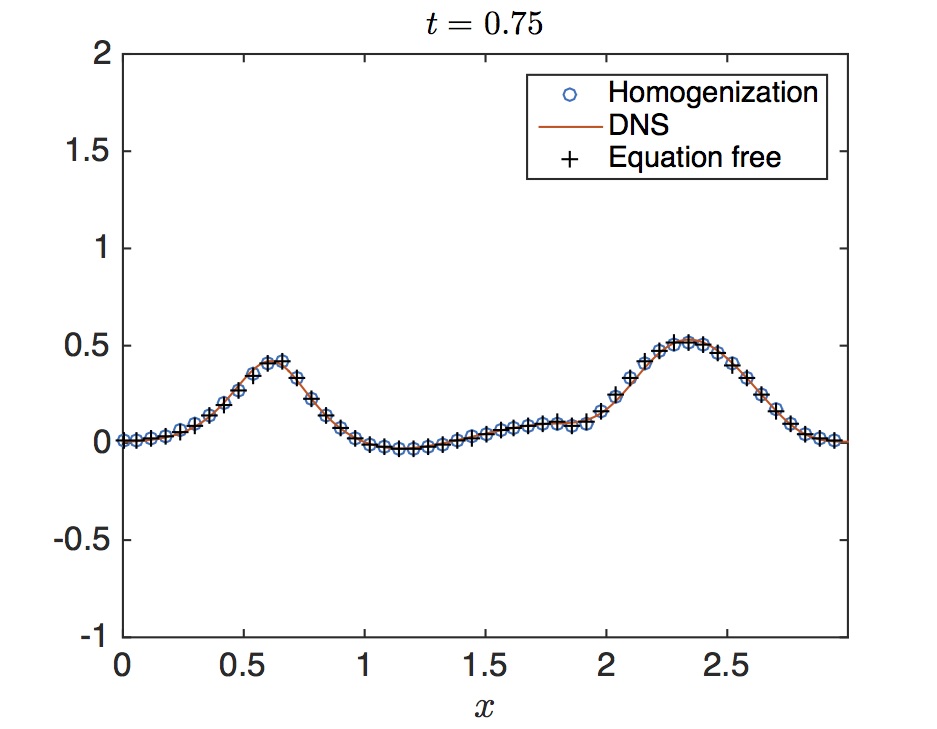}
         \includegraphics[width=0.49\textwidth]{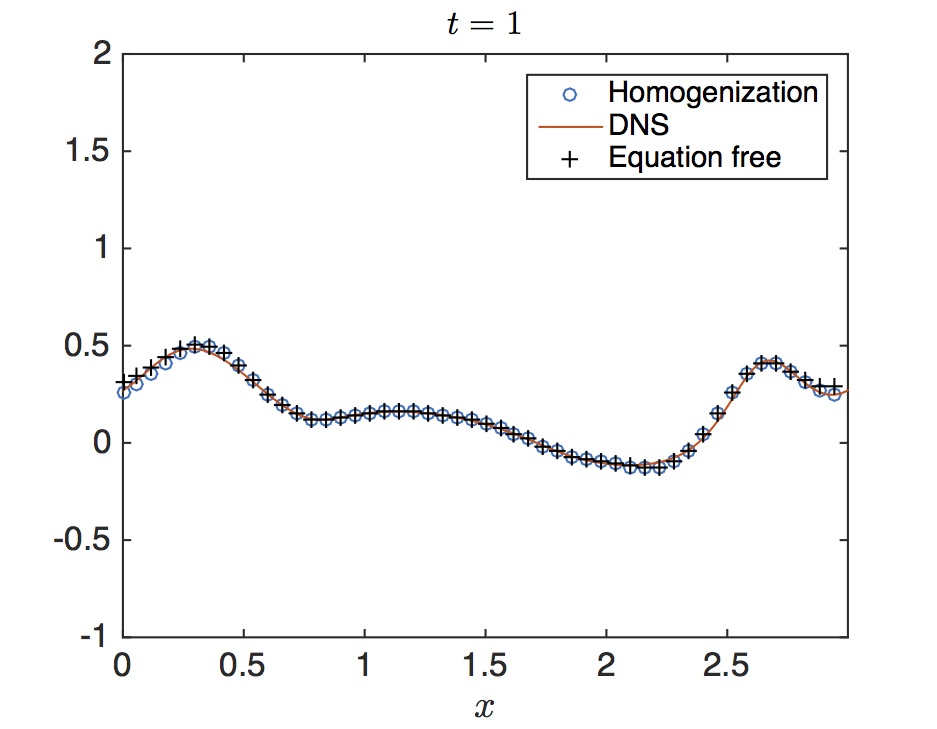}
	\caption[1D Periodic Media Solution]
	{The solution of the equation free approach (in a one-dimensional locally-periodic media) is compared with that of the homogenized and the direct numerical simulation. }
	\label{Fig:SolnPer1D}
\end{figure}

The theoretical results, in this paper, are valid  only for periodic coefficients, but the numerical algorithm is developed to treat more general coefficients. Here we give two one-dimensional examples where the wave speed is not periodic. In the Figure \ref{Fig:SolnPer1D}, we compare the solution of our multiscale algorithm with a direct numerical simulation, as well as  the homogenized solution. The initial condition for the wave equation is a Gaussian centred at the point $x = 1.5$, with a standard deviation $\sigma  = 0.08$. The coefficient function is locally-periodic and chosen as 
\begin{equation}
\label{Eqn_1d_LocPer_Coeff}
A^{\e}(x) = \left( 1.5  + \sin(2 \pi x) \right)\left( 1.5  + \sin(2 \pi x/\e) \right), \quad \text{ where }   \e  = 0.01.
\end{equation}
For a direct numerical simulation (DNS), the problem \eqref{Main_MultiscaleWave_Eqn} is solved over the space-time domain $(0,T] \times \Omega$, where $\Omega  = [0,3]$ and $T=1$, with periodic boundary conditions. The DNS uses $10$ points per wavelength (meaning $3000$ degrees of freedom in space), and the time-step is obtained by the CFL condition  $\frac{\sqrt{|A^{\e}|_{\infty}}  \delta t  }{\delta x} \leq 1$.  The equation free approach, however, uses only $50$ points in space (under-resolving the small scale variations) on the macro level, and a macroscopic time step $\triangle t = 0.01$. The parameters in the simulation of the equation free approach are chosen as $\eta = \tau  = 0.1$, and a kernel $K$ with $p=3$, and $q=5$ is used in the simulation. The homogenized solution, shown in Figure \ref{Fig:SolnPer1D}, uses the same discretization parameters as the macro solver in the equation free approach. The figure shows that the equation free approach captures the coarse scale part of the exact solution without resolving the $\e$-scale variations in the coefficient.

In Figure \ref{Fig:SolnNonPer1D}, we consider an example of yet another one-dimensional non-periodic media (known as almost-periodic media in the literature), modelled by the coefficient 
\begin{equation}
\label{Eqn_1d_AlmostPer_Coeff}
A^{\e}(x)  = \dfrac{1}{4} e^{\sin(2 \pi  \sqrt{2} x/\e)     +  \sin(2 \pi x/\e)}.
\end{equation} 
The equation free approach and the direct numerical simulations use precisely the same numerical parameters as in Figure \ref{Fig:SolnPer1D}, and the multiscale approach is again observed to accurately capture the coarse scale variations. In Figure \ref{Fig:1DSolnError}, the convergence (as $\e \to 0$) of the EFA to the homogenized solution for the locally periodic coefficient \eqref{Eqn_1d_LocPer_Coeff} and the almost periodic coefficient \eqref{Eqn_1d_AlmostPer_Coeff} is studied. Different $q,p$ pairings are used in the simulations and $\eta = 0.1$ for both simulations. As predicted by the theory, the convergence is of order $(\e/\eta)^{q+2}$ for the almost periodic case. For the locally periodic case, which is not covered by the theory in this paper, the same convergence rate is obtained by choosing higher values for the parameter $p$. Note that to be able to compute reference homogenized solution for the almost periodic case, we have approximated the irrational number $\sqrt{2}$ in the coefficient \eqref{Eqn_1d_AlmostPer_Coeff} by $1.41$, which makes the coefficient $100$-periodic (when $\e = 1$). We then compute the corresponding homogenized coefficient by 
$$
\left( \dfrac{1}{100} \int_{0}^{100} A^{1}(x)^{-1} \; d\bx \right)^{-1}  \approx 1.302004095265470
$$

\begin{figure}[h] 
	\centering
	\includegraphics[width=0.49\textwidth]{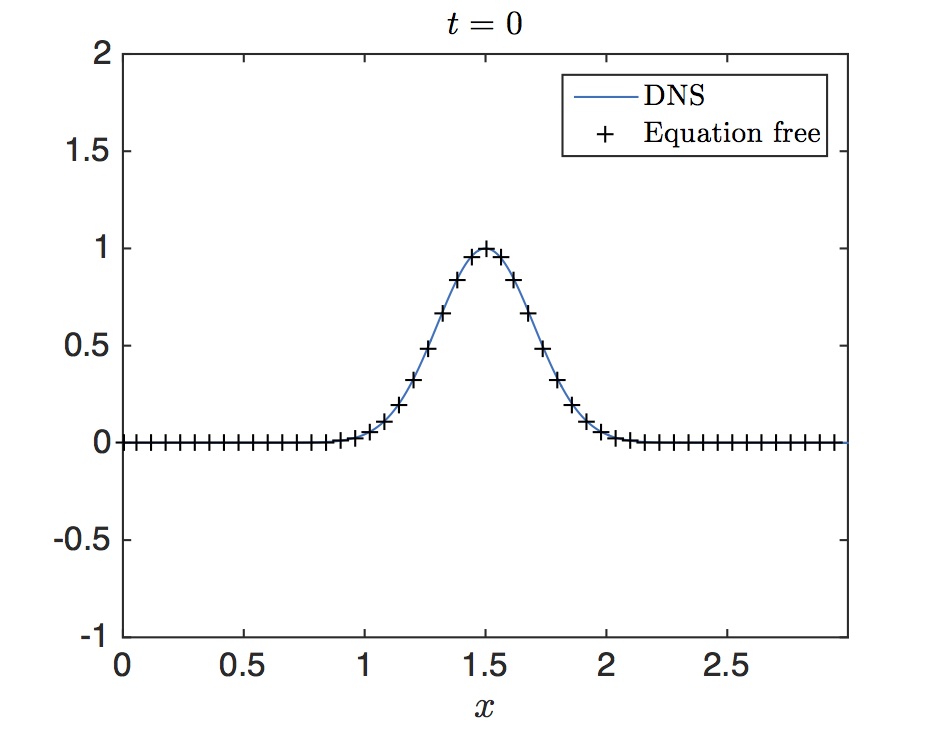}
		\includegraphics[width=0.49\textwidth]{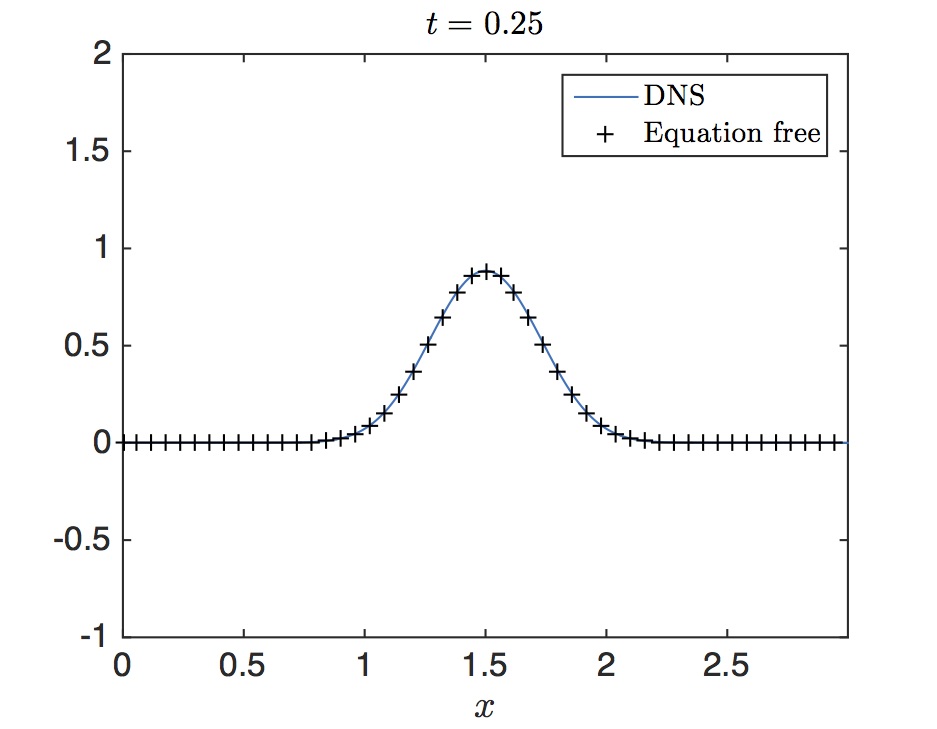}
	\includegraphics[width=0.49\textwidth]{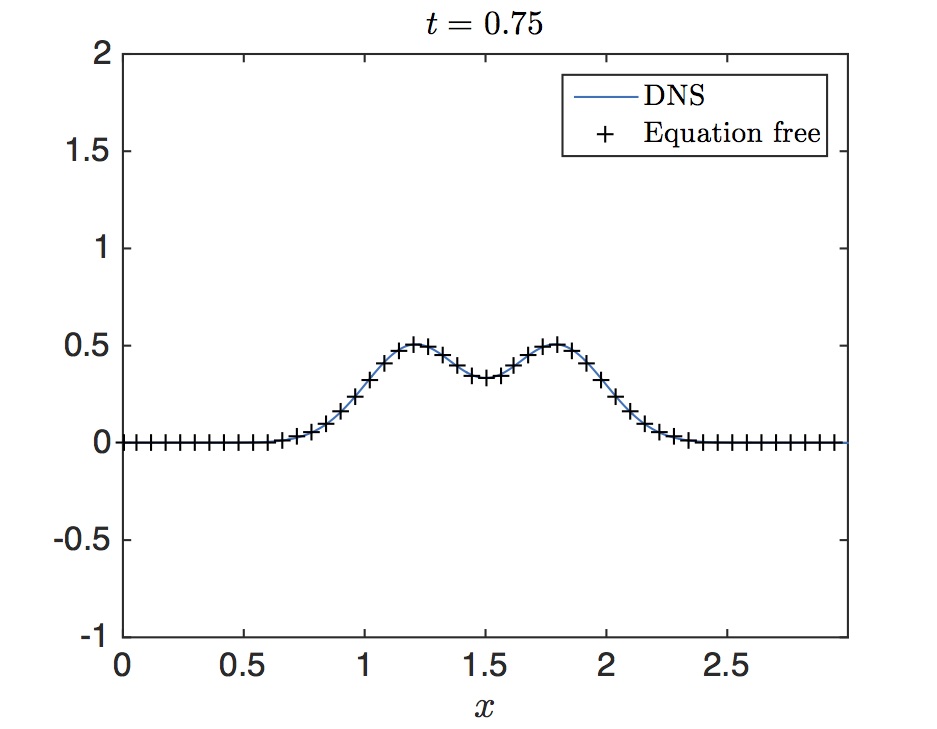}
         \includegraphics[width=0.49\textwidth]{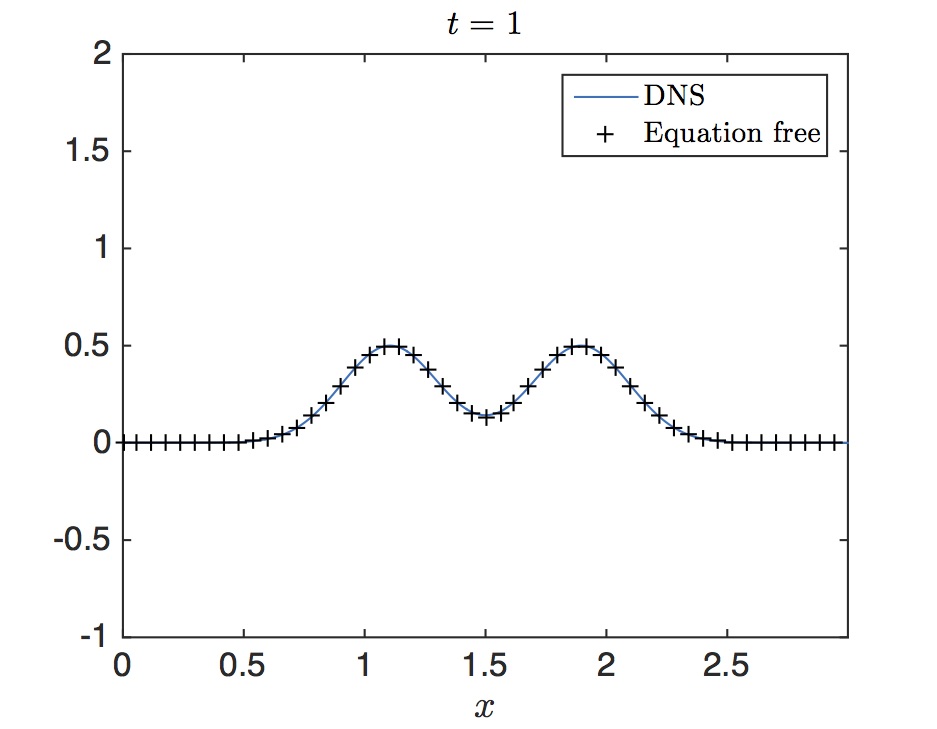}
	\caption[1D Non-periodic Media Solution]
	{The solution of the equation free approach (in a one-dimensional almost-periodic media) is compared with a direct numerical simulation. }  \label{Fig:SolnNonPer1D}
\end{figure}

\begin{figure}[ht] 
	\centering
	\includegraphics[width=0.49\textwidth]{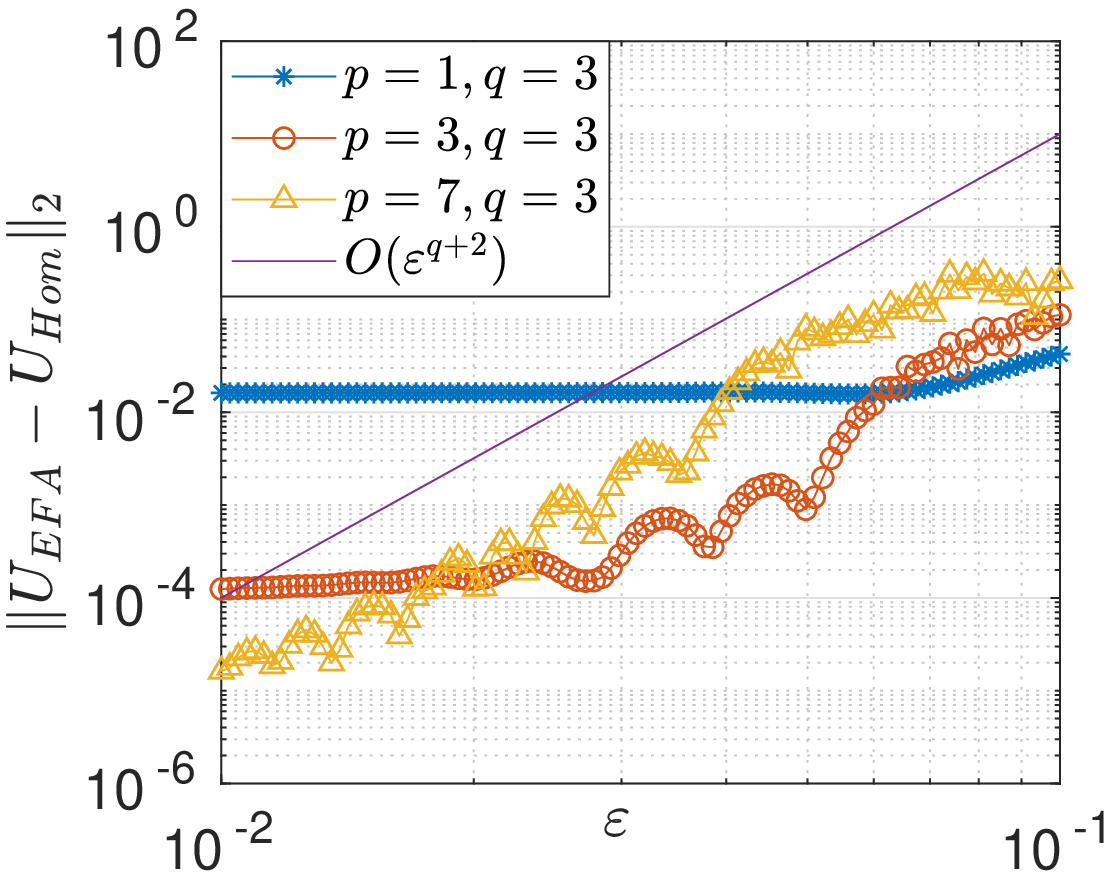}
		\includegraphics[width=0.49\textwidth]{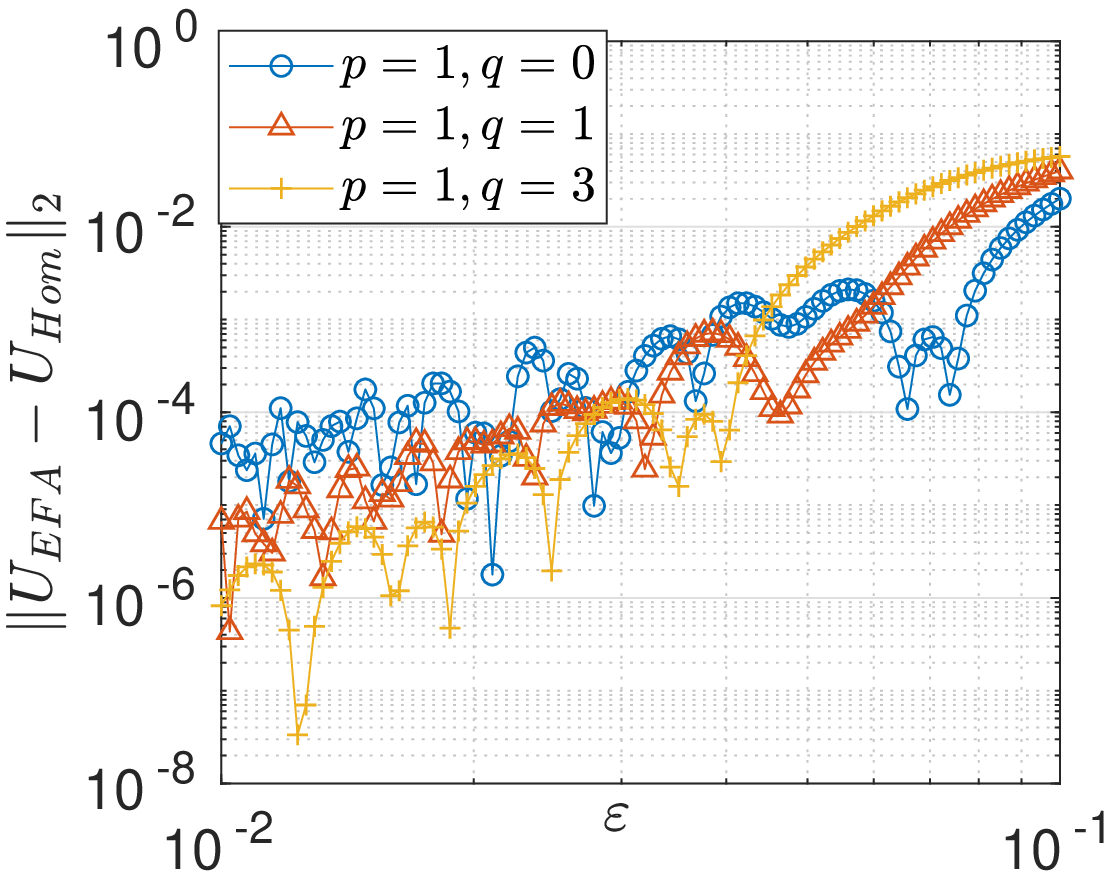}
	\caption[Error in solution]
	{The convergence (as $\e \to 0$) of the EFA to the homogenized solution for (left) the locally periodic coefficient \eqref{Eqn_1d_LocPer_Coeff} and (right) the almost periodic coefficient \eqref{Eqn_1d_AlmostPer_Coeff}. }
	\label{Fig:1DSolnError}
\end{figure}

\subsection{Solution in two dimensions}
\label{Sec:TwoDSoln}
To show the validity of the multiscale method in higher dimensions, we consider here a two-dimensional medium, where the coefficient function is non-periodic such that

\begin{equation} \label{2DCoeff_Anisotropic}
A^{\e}(\bx) = \dfrac{1}{3} \left( \dfrac{3}{2}  + \sin(2 \pi x_1/\e)\right) \left( \dfrac{3}{2}  + \dfrac{1}{2} \left(  \cos(2 \pi \sqrt{2}  x_1/\e)   + \cos(2 \pi x_2 /\e)  \right)  \right) D,
\end{equation}
where 
\begin{equation*}
D = 
\begin{bmatrix}
    1       & c \\
    c       & 1 \\
\end{bmatrix}.
\end{equation*}

Two cases (with different values for the constant $c$) are considered. In Figures \ref{Fig:SolnNonPer2D_Isotropic} and \ref{Fig:SolnNonPer2D_Anisotropic}, we consider an isotropic and an anisotropic material, which are modelled by $c=0$ and $c=1/2$ respectively. The solution of the equation free approach is compared to a direct numerical simulation $u^{\e}$, as well as a local average of $u^{\e}$ defined by $\left( \mathcal{K} u^{\e} \right)(t,\bx)$, see the Section \ref{HMM_Sec} for the definition of $\mathcal{K}$. The wave equation is solved over a time-space domain $[0,T] \times \Omega$, with $T=0.5$, and $\Omega = [0,1]^2$, and the small scale parameter is set to $\e = 0.025$. Periodic boundary conditions are used (on a macroscopic scale), and the initial data are assumed to be 
\begin{equation}\label{Eqn_InitialData2D}
u^{\e}(0,\bx) = \sin(2 \pi x_1) \cos(2 \pi x_2), \quad \partial_t u^{\e}(0,\bx)  = 1.
\end{equation}

The equation free approach uses $60 \times 60$ macroscopic points in space (under-resolving the small scale variations). Moreover, the parameter values $\eta = \tau  = 0.25$, $p= 5$, and $q=7$ are used for the simulation of the microscopic problem as well as the local averaging in the upscaling step. The DNS uses $10$ points per wavelength (meaning $400 \times 400$ points in space). Moreover, for both of the solvers (the equation free solver and the full multiscale problem) the time step is set such that the CFL condition $\sqrt{|A^{\e}|_{\infty}} \triangle t/\triangle x \leq 1$ holds with the largest possible time-step. The choice of the coefficient \eqref{2DCoeff_Anisotropic} is to test our multiscale algorithm for cases, where the theoretical assumptions of this paper do not hold. In other words, the theory in this work is based on the fact that the coefficient function is diagonal and periodic, but the simulations in this section include both non-periodic and non-diagonal examples. Moreover, the numerical simulations, depicted in Figures \ref{Fig:SolnNonPer2D_Isotropic} and \ref{Fig:SolnNonPer2D_Anisotropic}, show that the equation free approach captures the coarse features of the full multiscale solution even when the theoretical assumptions are not totally fulfilled. 

In Figure \ref{Fig:2DSolnError_c0_AlmostPer}, the convergence of the EFA approach to the homogenized solution for the two dimensional material coefficient \eqref{2DCoeff_Anisotropic} with $c=0$ is studied. Similar to the one dimensional case, different $(p,q)$ pairings are used and $\eta = 0.1$ is chosen in the simulation, and higher convergence rates are observed by taking higher values for $q$. To be able to compute the reference homogenized solution for the simulations in \ref{Fig:2DSolnError_c0_AlmostPer}, the irrational number $\sqrt{2}$ is approximated by $1.41$, which periodize the material coefficient with a period equal to $100$ in the $x_1$-variable (when $\e=1$). The homogenized coefficient is then computed by 
$$
A^{0} = \left( \dfrac{1}{100} \int_{0}^{100} \int_{0}^{1} A^{1}(\bx)^{-1} \; dx_2 dx_1 \right)^{-1}  \approx    0.485228277332784 \hspace{0.05cm} I.$$
Note that under this setting, the micro problems are still non-periodic, although the material coefficient is periodic on a large scale.

\begin{figure}[ht] 
	\centering
	\includegraphics[width=0.49\textwidth]{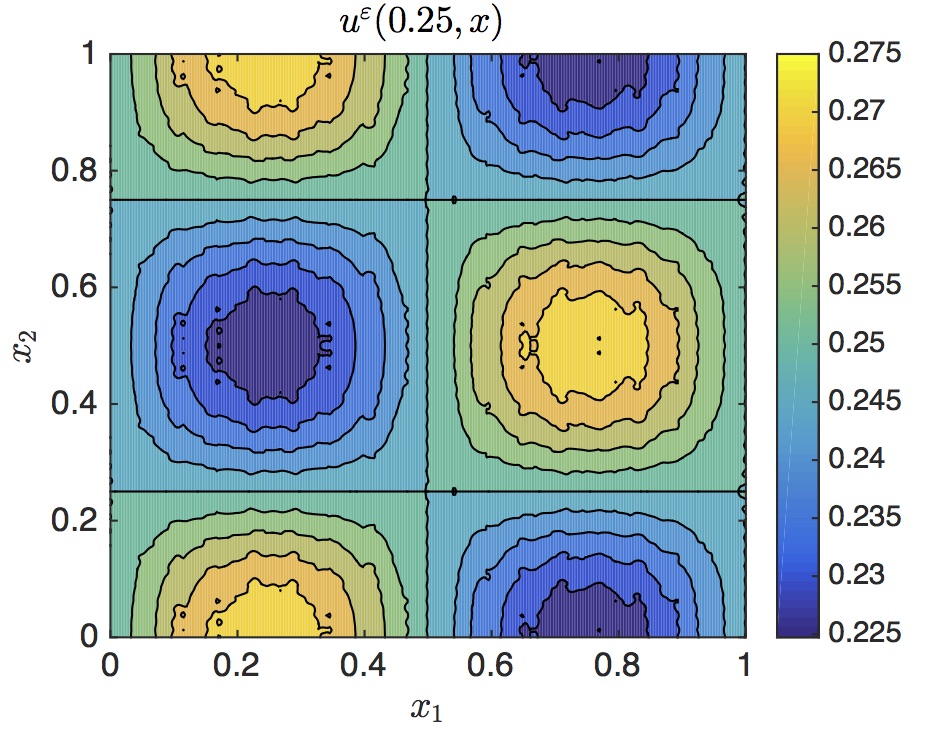}
		\includegraphics[width=0.49\textwidth]{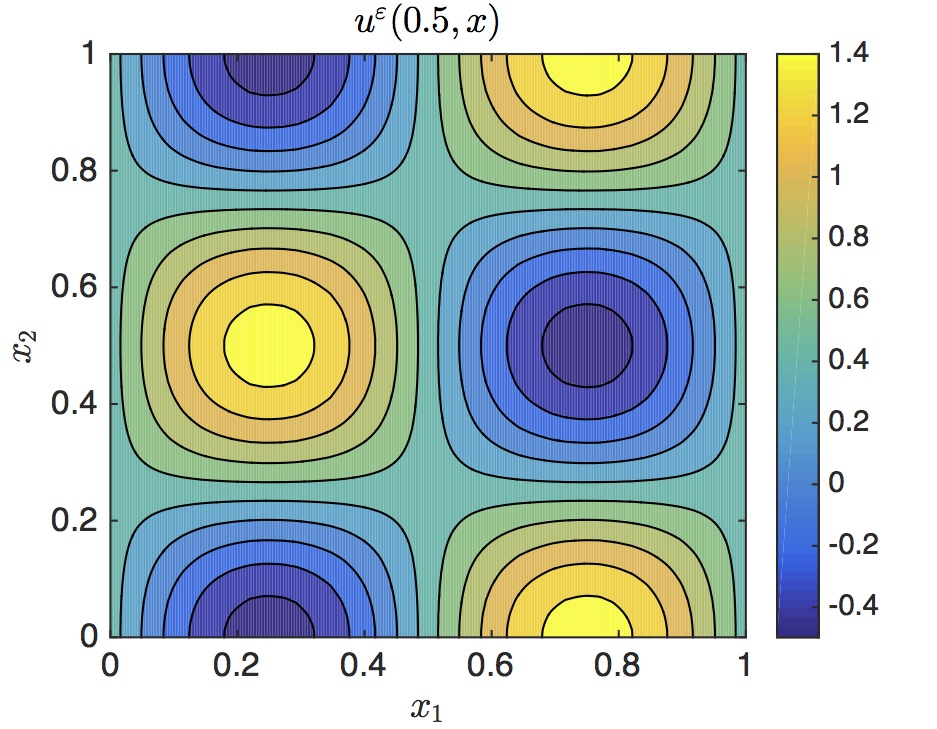}
	\includegraphics[width=0.49\textwidth]{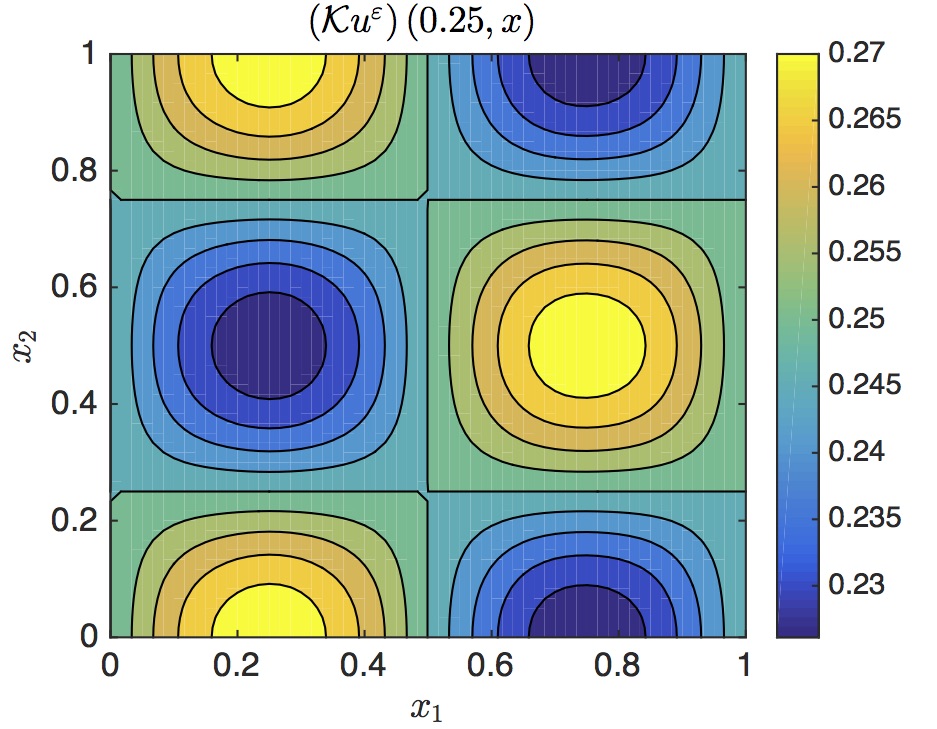}
         \includegraphics[width=0.49\textwidth]{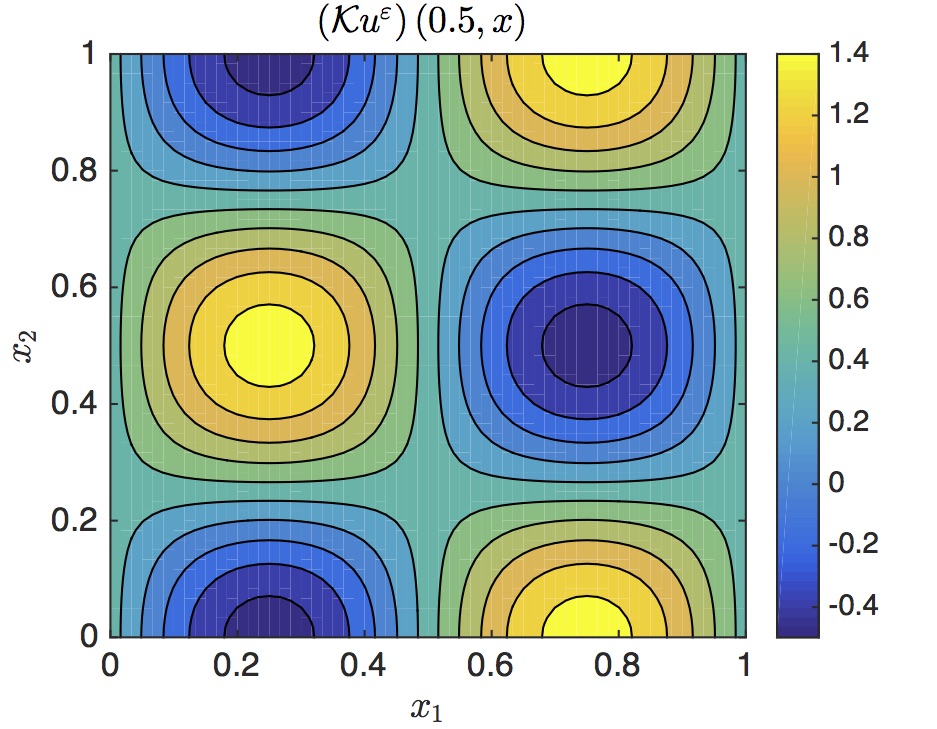}
         \includegraphics[width=0.49\textwidth]{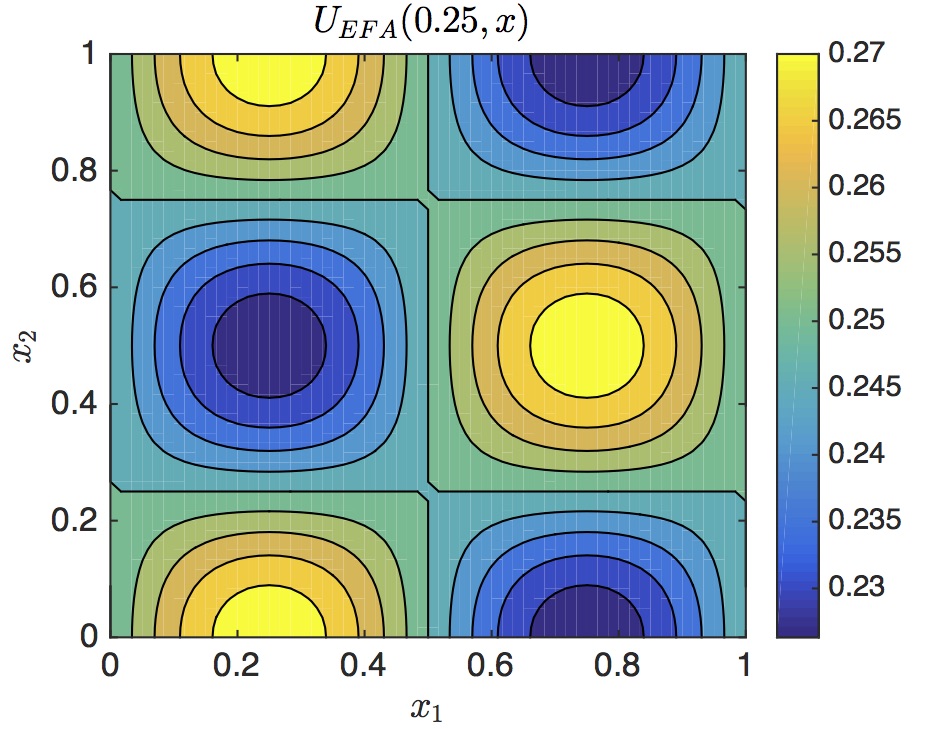}
         \includegraphics[width=0.49\textwidth]{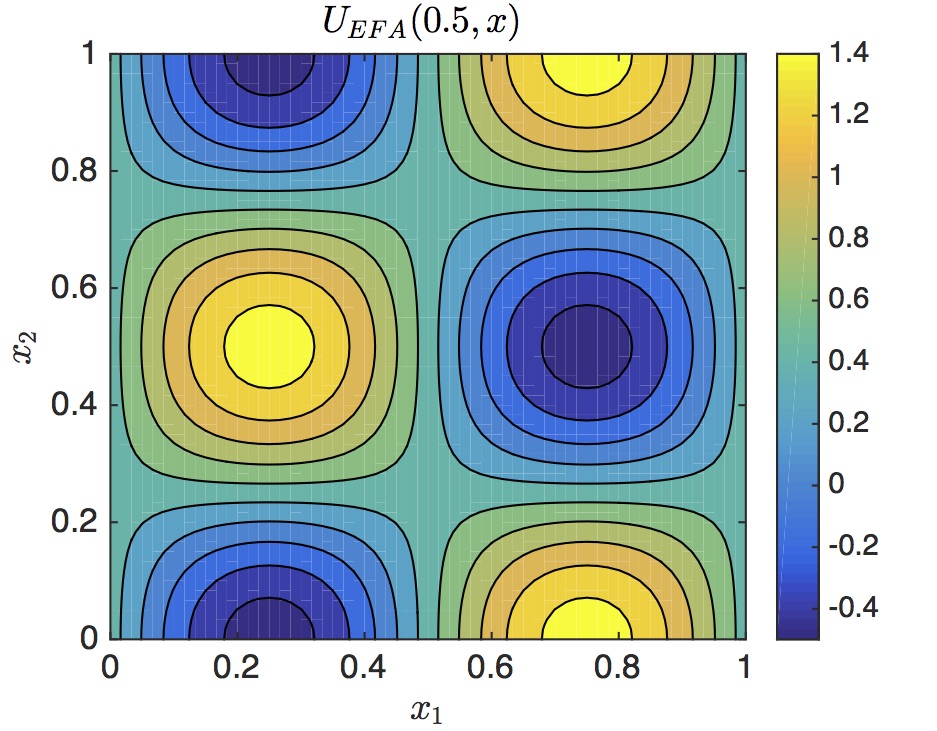}
	\caption[1D Non-periodic Media Solution]
	{Simulation of a non-periodic and isotropic material (when $c =0$ in \eqref{2DCoeff_Anisotropic}). (Top row) A direct numerical solution at times $t=0.25$ and $t = 0.5$. (Middle row) Local average $\left( \mathcal{K} u^{\varepsilon} \right)(t,\bx) $ at times $t=0.25$ and $t=0.5$. (Bottom row) The solution of the equation free approach at times $t=0.25$ and $t = 0.5$. Note that the small scale oscillations for the solution of the DNS (top row) at $t=0.25$ can be identified on the picture, but the solution of the equation free approach captures only the large scale behaviour.}
	\label{Fig:SolnNonPer2D_Isotropic}
\end{figure}

\begin{figure}[ht] 
	\centering
	\includegraphics[width=0.49\textwidth]{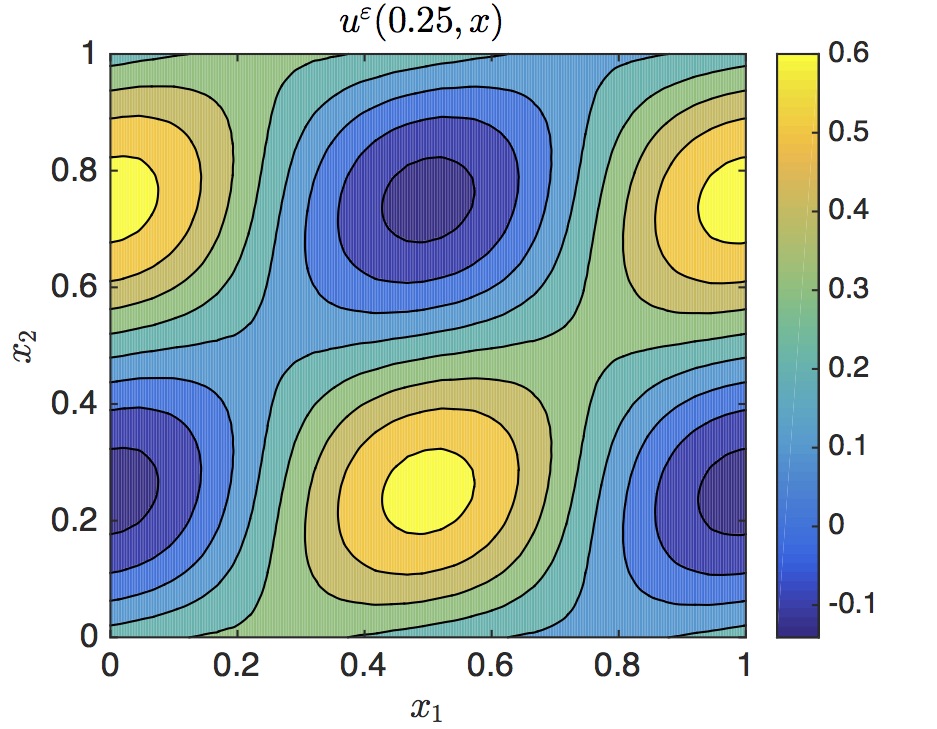}
		\includegraphics[width=0.49\textwidth]{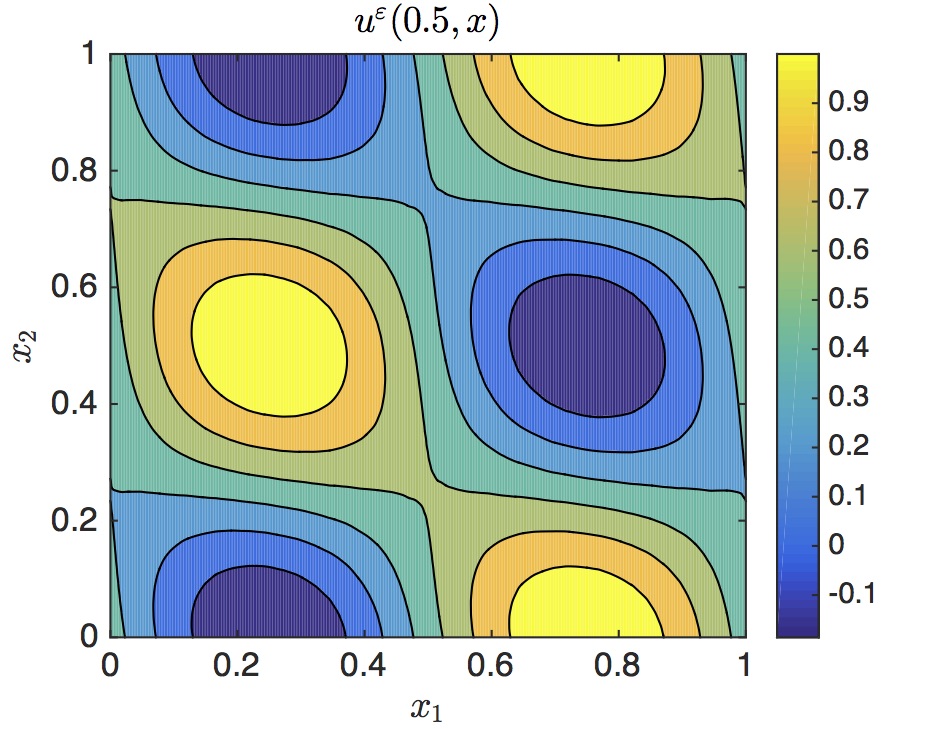}
	\includegraphics[width=0.49\textwidth]{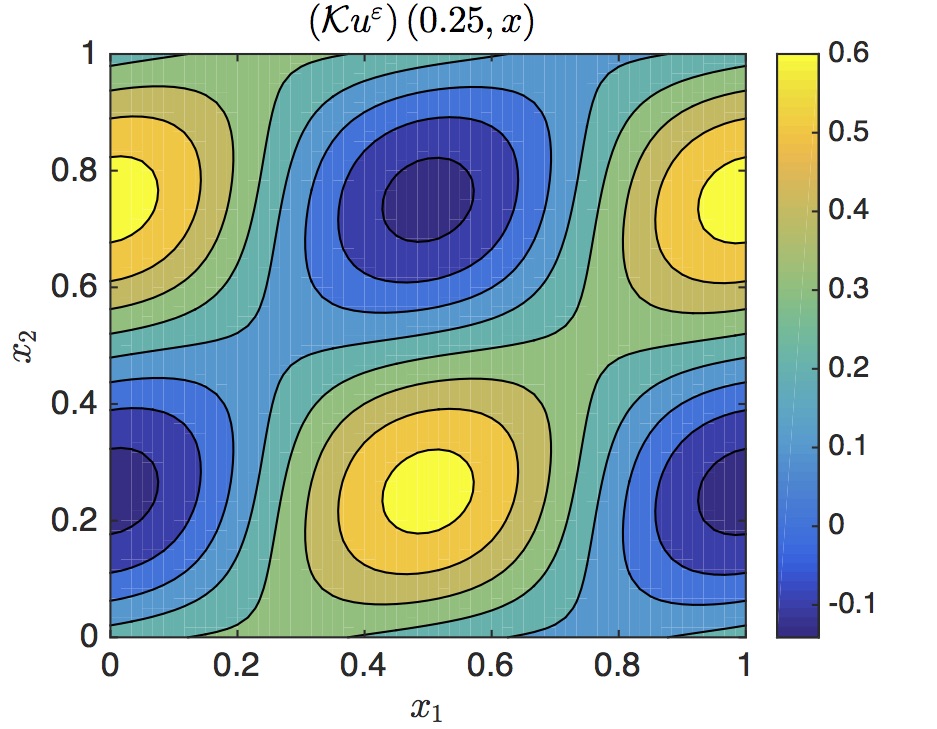}
         \includegraphics[width=0.49\textwidth]{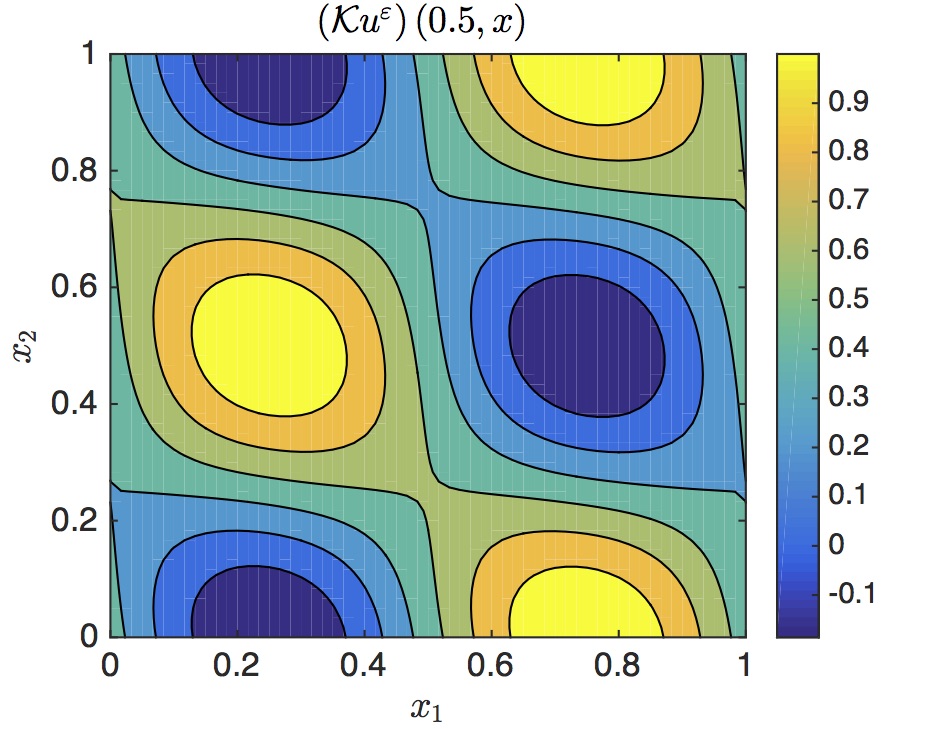}
         \includegraphics[width=0.49\textwidth]{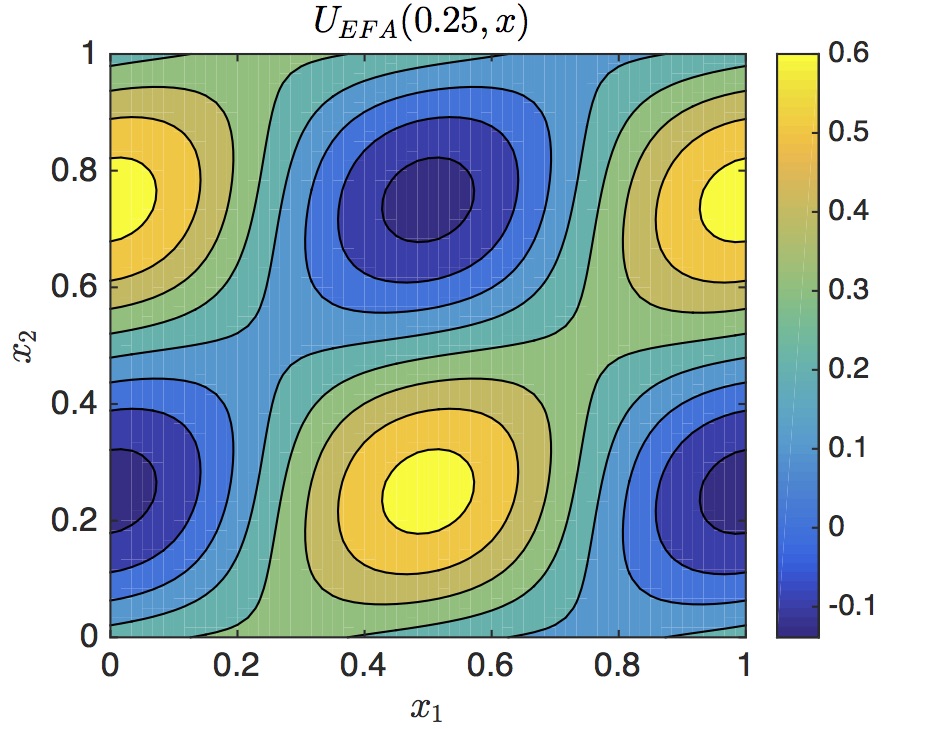}
         \includegraphics[width=0.49\textwidth]{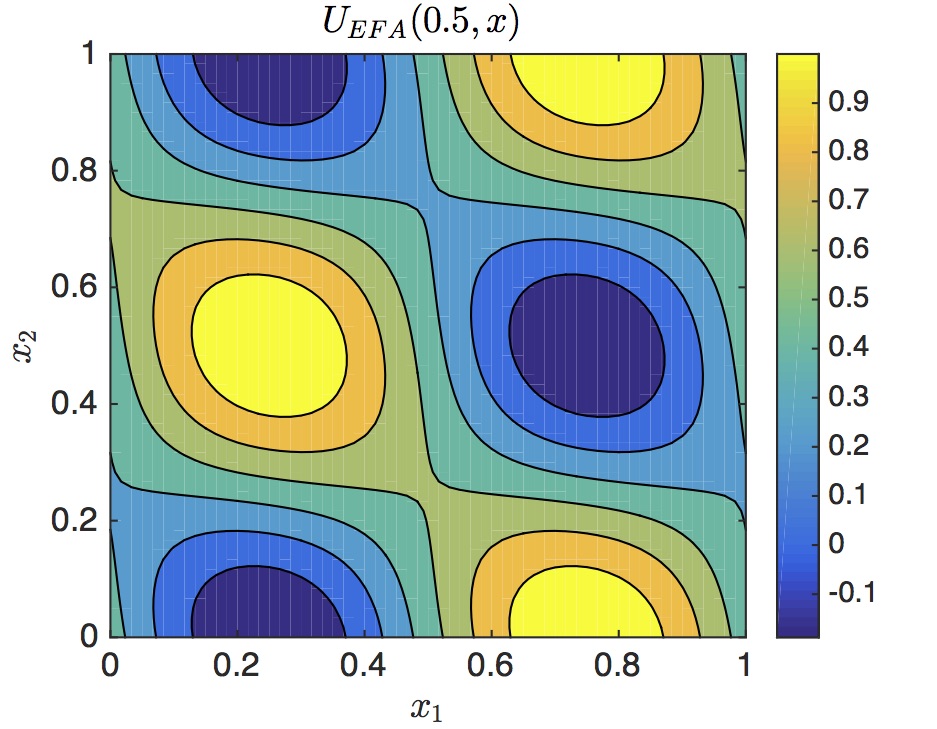}
	\caption[1D Non-periodic Media Solution]
	{Simulation of an isotropic and non-periodic  material. (Top row) A direct numerical solution at times $t=0.25$ and $t = 0.5$. (Middle row) Local average $\left( \mathcal{K} u^{\varepsilon} \right)(t,\bx) $ at times $t=0.25$ and $t=0.5$. (Bottom row) The solution of the equation free approach at times $t=0.25$ and $t = 0.5$. }
	\label{Fig:SolnNonPer2D_Anisotropic}
\end{figure}

\begin{figure}[ht] 
	\centering
	\includegraphics[width=0.6\textwidth]{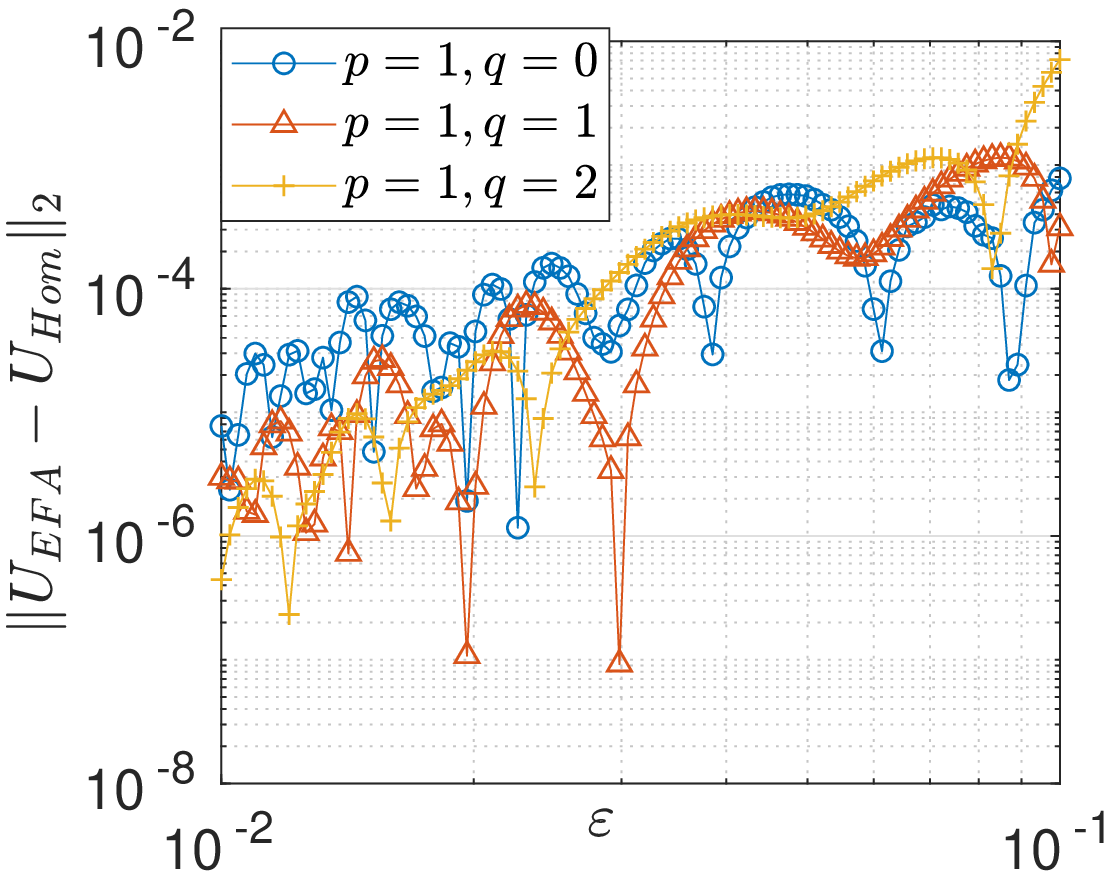}
	\caption[Error in solution]
	{The convergence (as $\e \to 0$) of the EFA to the homogenized solution for  the almost periodic coefficient \eqref{2DCoeff_Anisotropic}. }
	\label{Fig:2DSolnError_c0_AlmostPer}
\end{figure}


\section{Conclusion}
In this article, a numerical method for multiscale wave propagation problems in non-divergence form is proposed and analysed. Multiscale methods based on the HMM framework were previously developed and analysed for wave propagation problems in divergence form, see e.g., \cite{Engquist_Holst_Runborg_1,Arjmand_Runborg_3,Abdulle_Grote_1}. Multiscale wave equations in non-divergence form have homogenised limits, which are different from the limiting equations for wave equations in divergence form. This motivates the need for developing new multiscale methods, which include additional modifications  in comparison to the HMM type methods. Moreover, for the wave equations in divergence form, the analysis of the HMM type algorithms typically relies on the symmetry properties of the operator $-\nabla \cdot A^{\e} \nabla$, which is missing in the theoretical setup of the current study. In this paper, an analysis for the error between the homogenised limit and the solution of the EFA is given, and numerical evidence corroborating the theoretical results are shown.We note that the extension of the method for random media requires an additional Monte Carlo step to be included in the algorithm. However, this extension is conceptually no different than existing works on HMM for random media, see e.g., \cite{Abdulle_Barth_Schwab_13}.

\clearpage


\bibliographystyle{plain}
\bibliography{references}

\end{document}